\documentclass[final,leqno,onefignum,onetabnum]{siamltex1213}
\usepackage{amssymb,amsmath,amsfonts,amstext,amscd}%
\usepackage[english]{babel}%
\usepackage{graphicx,xcolor,subfig,multirow}
\newcommand{\sfac}{0.49\textwidth}%
   
\title{Computable error estimates for finite element approximations of
  elliptic partial differential equations with rough stochastic
  data\thanks{The research was supported by the Swedish Research
    Council grant VR-621-2014-4776, and the Swedish e-Science Research
    Center. It was carried out while E.J.~Hall was a G\"oran
    Gustafsson postdoctoral fellow at KTH Royal Institute of
    Technology. R.~Tempone is a member of the KAUST Strategic Research
    Initiative, Center for Uncertainty Quantification in Computational
    Sciences and Engineering. R.~Tempone received support from the
    KAUST CRG3 Award Ref: 2281. H.~Hoel received support by Norges
    Forskningsr{\aa}d, research project 214495 LIQCRY.}}%
\author{Eric Joseph Hall\footnotemark[2]%
  \and H\r{a}kon Hoel\footnotemark[3]%
  \and Mattias Sandberg\footnotemark[4]%
  \and Anders Szepessy\footnotemark[4]%
  \and Ra\'ul Tempone\footnotemark[5]}%

\newtheorem{remark}{\emph{Remark}} 

\newcommand{\RR}{\mathbf{R}}%
\newcommand{\NN}{\mathbf{N}}%
\DeclareMathOperator{\mexp}{\mathbf{E}}%
\newcommand{\prob}{P}%
\newcommand{\ind}[1]{\boldsymbol{1}_{ #1}}
\DeclareMathOperator{\cov}{Cov}%
\newcommand{\as}{\prob\text{-}\mathrm{a.s.}}%
\DeclareMathOperator{\divergence}{div}%
\DeclareMathOperator{\dd}{d\!}%
\newcommand{\abs}[1]{\left| #1 \right|}%
\newcommand{\E}[1]{\mathbf{E}\left[#1\right]}%
\newcommand{\pr}[1]{\left( #1 \right)}%

\begin{document}
\maketitle%
\slugger{sisc}{xxxx}{xx}{x}{x--x}%

\renewcommand{\thefootnote}{\fnsymbol{footnote}}%

\footnotetext[2]{Department of Mathematics and Statistics, University
  of Massachusetts Amherst, Amherst, MA 10030, United States
  (\texttt{hall@math.umass.edu}).}%
\footnotetext[3]{Department of Mathematics, University of Oslo, 0316
  Oslo, Norway (\texttt{haakonah@math.uio.no}).}%
\footnotetext[4]{Department of Mathematics, KTH Royal Institute of
  Technology, 100 44 Stockholm, Sweden (\texttt{msandb@kth.se},
  \texttt{szepessy@kth.se}).}%
\footnotetext[4]{Applied Mathematics and Computational Science, King
  Abdullah University of Science and Technology, Thuwal, 23955-6900,
  Kingdom of Saudi Arabia. (\texttt{raul.tempone@kaust.edu.sa}).}%

\renewcommand{\thefootnote}{\arabic{footnote}}%

\begin{abstract}
  We derive computable error estimates for finite element
  approximations of linear elliptic partial differential equations
  (PDE) with rough stochastic coefficients. In this setting, the exact
  solutions contain high frequency content that standard \emph{a
    posteriori} error estimates fail to capture. We propose
  goal-oriented estimates, based on local error indicators, for the
  pathwise Galerkin and expected quadrature errors committed in
  standard, continuous, piecewise linear finite element
  approximations. Derived using easily validated assumptions, these
  novel estimates can be computed at a relatively low cost and have
  applications to subsurface flow problems in geophysics where the
  conductivities are assumed to have lognormal distributions with low
  regularity. Our theory is supported by numerical experiments on test
  problems in one and two dimensions.
\end{abstract}



\pagestyle{myheadings}%
\thispagestyle{plain}%
\markboth{E.~J.~HALL, H.~HOEL, M.~SANDBERG, A.~SZEPESSY AND
  R.~TEMPONE}{COMPUTABLE ESTIMATES FOR PDE WITH ROUGH STOCHASTIC
  DATA}%

%
\section{Introduction}
\label{sec:introduction}
%

There is a vast body of literature concerning \emph{a posteriori}
error estimation of finite element approximations of elliptic PDE. All
of these works, to some extent, examine models in which the
conductivities enjoy a certain degree of smoothness.  Motivated by
problems arising in geophysics, we present computable error estimates
for piecewise linear finite element approximations of a class of
elliptic problems where the conductivities are modeled by \emph{rough}
stochastic processes. In this setting, the typical \emph{a posteriori}
error analysis does not yield computable estimates.

To understand where the standard analysis breaks down, it is
instructive to consider the following homogeneous two-point boundary
value problem:
\begin{equation}
  \label{eq:model-problem}
  -(a(\omega,x) u^\prime(\omega,x))^\prime = 0 
\end{equation}
$\as$ for $x \in [0,1]$ subject to the boundary conditions
$u(0) = u(\omega,0) = 0$ and
$a(1)u^\prime(1) = a(\omega,1)u^\prime(\omega,1) = 1$, $\as$, where
$\omega \in \Omega$ for a given probability space
$(\Omega, \mathbb{F}, \prob)$ and where $u^\prime$ denotes the
derivative of $u$ with respect to $x$. We seek a solution,
$u = u(\omega,\cdot)$, to this equation, parameterized by $\omega$, in
the Sobolev space $\mathcal{H}^{1}(0,1)$ for a given process
$a(x) = a(\omega,x) = e^{B(\omega,x)}$ where
$(B(\omega,x))_{x \in [0,1]}$ is a Brownian bridge pegged to zero at
the end points. By Kolmogorov's continuity theorem
\cite{KaratzasShreve:1991}, $B(x)$ admits a version that is H\"{o}lder
continuous with exponent $1/2-\varepsilon$ for
$\varepsilon \in (0, 1/2)$ and thus $a$ admits a version
$a \in C^{\frac{1}{2}-\varepsilon}([0,1])$. In what follows we take
$a$ to be this almost $1/2$-H\"{o}lder continuous version and we note
that $a(0) = a(1) = 1$ for almost all $\omega \in \Omega$. As a matter
of course, one cannot expect much smoothness from the solution: we
have $ u \in C^{\frac{3}{2} - \varepsilon}([0,1])$, for
$\varepsilon \in (0,1/2)$, since $u^\prime(x) = 1/a(x)$.

Our goal is to determine moments of an observable $(u, g)$, a linear
functional of the solution to \eqref{eq:model-problem}, for a given
$g$. The problem is to estimate the expected error in the observable,
\begin{equation*}
  \label{eq:error-functional2}
  \mexp \left[ (u - \bar u_h, g) \right] = \mexp\left[ \int_0^1 (u-\bar u_h)(x) g(x) \dd x \right],
\end{equation*}
where $\bar u_h$ is the finite element approximation with a given
quadrature of the bilinear form, for a given $\omega$, taking values
in $V_h$, the space of piecewise linear elements.  The error can be
split into a part corresponding to the Galerkin error
\begin{equation}
  \label{eq:error-functional}
  \mexp \left[ (u -  u_h, g) \right] = \mexp\left[ \int_0^1 (u- u_h)(x) g(x) \dd x \right],
\end{equation}
and another part corresponding to the quadrature error
\begin{equation*}
  \label{eq:error-functionalquad}
  \mexp \left[ (u_h - \bar u_h, g) \right] = \mexp\left[ \int_0^1 (u_h-\bar u_h)(x) g(x) \dd x \right],
\end{equation*}
where $u_h$ solves the finite element problem with exact
quadrature. In the case of finite element approximations of problems
with smooth conductivity functions $a$, the quadrature error can be
made small compared to the Galerkin error with asymptotically negligible
additional work.  In the case of conductivities with low regularity,
the situation is different: we observe that, under mild assumptions,
the pathwise quadrature error is typically much larger than the
Galerkin error and the expected value of the quadrature error, proved
in Theorem \ref{thm:quadScales}, is of the same order as the Galerkin
error, in the sense that doubling the number of quadrature points
reduces the expected quadrature error with the same factor as the
Galerkin error when doubling the number of nodal points. Therefore the
quadrature error is at least as important as the Galerkin error when
analyzing the discretization error for problems where the conductivity
has low regularity. Indeed, if the regularity is lower than the rough
lognormal conductivity considered in \eqref{eq:model-problem}, then
the quadrature error may dominate. The pathwise Galerkin error,
derived in Theorem \ref{thm:GalerkinScales}, is in some sense simpler
to analyze. We observe that the pathwise Galerkin error is of the same
order as the expected value of the Galerkin error, suggesting that
there are no large stochastic cancellations present, in contrast to
the situation for the quadrature error. Nevertheless, the low
regularity of the conductivity makes our analysis of the Galerkin
error different from the typical \emph{a posteriori} analysis.

To motivate why the standard \emph{a posteriori} analysis of Galerkin
errors is not directly applicable in the case of conductivities with
low regularity, we sketch the classical \emph{a posteriori} error
analysis in the energy norm for the model problem
\eqref{eq:model-problem}. We begin by recalling a classical result in
\cite{Babuska:1971sn} that gives an \emph{a priori} bound for the
finite element approximation in the energy norm,
\begin{equation*}
  \| w \|_{E} := \left( \int_0^1 a(x) (w^\prime(x))^2 \dd x \right)^{\frac{1}{2}},
\end{equation*}
for elliptic PDE (analogous estimates for elliptic PDE with lognormal
coefficients, such as those given in \cite{Charrier:2012,
  CharrierScheichlTeckentrup:2013}, are discussed further in \S
\ref{sec:model}).  Namely, for $s < 1/2$,
\begin{equation}
  \label{eq:a-priori-est}
  \| u - u_h \|_E \leq \inf_{v \in V_h} \| u - v \|_E \lesssim h^s \|u\|_{1+s}
\end{equation}
estimates the Galerkin error committed pathwise assuming that the
quadrature used in computing $u_h$ is exact. Here $\| \cdot \|_{1+s}$
denotes the norm in the Sobolev space $\mathcal{H}^{1+s}(0,1)$ of
functions with up to $1+s$ derivatives in $L^2(0,1)$ (for a definition
of such spaces see \cite{ScottBrenner:2008}) and we observe that
$u \in C^{\frac{3}{2}-\varepsilon}([0,1])$ so that also
$u\in\mathcal{H}^{\frac{3}{2}-\varepsilon}(0,1)$.  Next, we rewrite
the pathwise error functional for \eqref{eq:error-functional} as
\begin{equation*}
  \int_0^1 a(x)(u^\prime - u_h^\prime)(x)
  (\lambda^\prime - \lambda_h^\prime)(x) \dd x
\end{equation*}
by introducing the dual problem in the variable $\lambda$:
\begin{equation*}
  -(a(\omega,x)\lambda^\prime(\omega,x))^\prime = g(x),
\end{equation*}
$\as$ for $x \in [0,1]$, subject to the boundary conditions
$\lambda(\omega,0) = \lambda^\prime(\omega,1) = 0$, $\as$ The
convergence of the pathwise Galerkin error is on the order
\begin{equation*}
  (u - u_h, g) \leq \| u - u_h \|_E \| \lambda - \lambda_h \|_E 
  = O(h^{2s}),
\end{equation*}
using \eqref{eq:a-priori-est} for sufficiently regular $g$. The error
in the energy norm is
\begin{equation}\label{u_h_residual_jump}
  \begin{split}
    \|u - u_h\|_E^2 &=  \int_0^1 a(u^\prime - u_h^\prime) (u - u_h)^\prime\dd x = \int_0^1 a(u^\prime - u_h^\prime) (u - u_h-v)^\prime\dd x\\
    &=\sum_{k=0}^{N-1} \int_{x_k}^{x_{k+1}}a(u'-u_h')(u-u_h-v)' \dd x\\
    &=\sum_{k=0}^{N-1} \Big(\int_{x_k}^{x_{k+1}} -(a(u'-u_h'))'(u-u_h-v) \dd x\\
    &\quad+ [a(u'-u_h')(u-u_h-v)]_{x_k}^{x_{k+1}} \Big)\\
    &= \sum_{k=0}^{N-1} \Big(\int_{x_k}^{x_{k+1}} (\underbrace{(-au')'}_{=0}+(au_h')')(u-u_h-v) \dd x\\
    &\quad+ [a(u'-u_h')(u-u_h-v)]_{x_k}^{x_{k+1}} \Big)\\
    &=\sum_{k=0}^{N-1} \int_{x_k}^{x_{k+1}}(au_h')'(u-u_h-v) \dd x\\
    &\quad +\sum_{k=1}^{N-1}
    a(x_k)\underbrace{(u_h'(x_k^+)-u_h'(x_k^-))(u-u_h-v)(x_k)}_{(u_h'(x_k^+)-u_h'(x_k^-))\langle\delta_{x-x_k},u-u_h-v\rangle}
    \\
    &\quad + a(1)(u'(1)-u_h'(1))(u-u_h-v)(1)\\
    &\quad -a(0)(u'(0)-u_h'(0))(u-u_h-v)(0)\\
    & =\int_0^1 (a u_h^\prime)^\prime (u - u_h-v)\dd x +\left(1 -
      u_h^\prime(1)\right) \left(u(1) - u_h(1)-v(1)\right),
  \end{split}
\end{equation}
for any $v \in V_h$ introduced by the Galerkin orthogonality. The
terms $u_h'(x_k^-)$ and $u_h'(x_k^+)$ denote left and right limits,
respectively, and $\langle\delta_{x-x_k},u-u_h-v\rangle$ denotes the
action of the Dirac delta distribution on the function $u-u_h-v$. This
operation is well defined, since $u-u_h-v$ is continuous.  The
residual, viewed in the sense of distributions, has two parts. One
part, $au_h''$ and $1-u_h'(1)$, consists of point masses at the nodal
points, due to the jump of the piecewise constant derivative $u_h'$,
and the residual of the weak boundary condition, respectively. This
part vanishes by choosing the test function $v=\pi(u-u_h)$ as the
nodal interpolant of $u-u_h$, so that $u-u_h-v$ is zero in all nodal
points. The remaining part of the residual $R:=a'u_h'$ does not
vanish, and by evaluating the integral as the duality pairing
$\int_0^1R(u-u_h-v)\dd x=\langle R,u-u_h-v\rangle$ between the Sobolev
spaces $\mathcal{H}^{-s}$ and $\mathcal{H}^s$ we obtain
\begin{equation*}
  \|u - u_h\|_E^2 \leq \| R \|_{-s} \| u - u_h - v\|_{s}.
\end{equation*}
Then the standard \emph{a posteriori} error estimate in the energy
norm, (\cite{BabuskaRheinboldt:1978ee, BabuskaRheinboldt:1981ap}), is
\begin{equation*}
  \| u - u_h \|_E \lesssim h^{1-s} \|R \|_{-s},
\end{equation*}
since $\| u - u_h - \pi(u - u_h)\|_{s} \lesssim h^{1-s} \|u - u_h\|_E$
by \eqref{eq:a-priori-est}. Since $R=a'u_h'=B'e^{B}u'_h$ has the
regularity of white noise in space, we have $\| R \|_{L^2} = \infty$
for $s = 0$. For other values of $s$, the negative Sobolev norm of the
residual is not easily computed.  Consequently, this direct
application of the standard \emph{a posteriori} error analysis does
not provide an estimate that can be easily computed even for the
preceding simple model problem.

This paper proposes computable, goal-oriented estimates for the
Galerkin and quadrature errors committed in standard, continuous,
piecewise linear finite element approximations of elliptic PDE with
rough stochastic conductivities. A key insight for these estimates
comes from studying the frequency content of the components of the
Galerkin error functional, arising from the dual weighted residual
analysis, for the simple model problem \eqref{eq:model-problem}. In
contrast to the case of smooth conductivities, we see that the
components' high-frequency content, which cannot be computed directly,
is non-negligible. Nevertheless, this high-frequency contribution can
be approximated by low-frequency content. This formal frequency study
suggests an assumption on scales that yields an estimator, based on
local error indicators, for observables of the pathwise Galerkin error
and hence the expected Galerkin error. Estimators for the expected
quadrature error are also obtained under an assumption on scales where
the stochastic cancellation effects present in the quadrature error
are analyzed using tools from stochastic sensitivity analysis.

\emph{A posteriori} estimates provide a quantitative measure of the
quality of a numerical experiment. There is a rich literature
concerning \emph{a posteriori} error estimation for deterministic
elliptic problems (see, for example,
\cite{ErikssonEtAl:1995,GilesSuli:2002,BangerthRannacher:2003,
  AinsworthOden:1997}). While our error estimates rely on the computed
solution to the finite element approximation and not on the analytic
solution, we call our estimates ``computable'' as opposed to \emph{a
  posteriori} as they require information about the regularity of the
analytic solution. Here, this requisite information appears in the
form of decay assumptions on the approximate solutions, which we call
assumptions on scales. For the models with rough lognormal
conductivity considered here, these decay rates are satisfied for all
the linear observables under consideration as the given covariances
are Lipschitz continuous. In principle our analysis can also treat
conductivities with less regularity, such as conductivities whose
covariances are not Lipschitz. In such cases, the decay assumptions
would need to be obtained heuristically and might also depend on the
particular choice of observable. We do not present a method for
obtaining such rates through strictly \emph{a posteriori}
information. In spite of these technicalities, the estimates we
propose are computable and can be useful when the standard \emph{a
  posteriori} analysis fails and qualitative assessments based solely
on \emph{a priori} information might be misleading.

The proposed estimators could be used in constructing adaptive
algorithms for variance reduction techniques for Monte Carlo (MC)
methods. For variance reduction techniques, such as the Multilevel
Monte Carlo (MLMC) method \cite{Giles:2008}, Multi-Index Monte Carlo
(MIMC) method \cite{Haji-AliNobileTempone:2014}, and Continuation MLMC
(CMLMC) method \cite{CollierEtAl:2014}, these novel, computable error
estimates can be used to determine final level stopping criterion. An
extension of the present theory would render it possible to also
estimate mean square errors between numerical solutions on consecutive
resolution levels. Since the mean square error implicitly bounds the
variance, this may give rise to highly efficient \emph{a posteriori}
adaptive variance reduction methods for MLMC; see
\cite{HoelHappolaTempone:2014aa} for an application of such ideas in
the setting of low regularity stochastic differential equations
(SDE). Applications of these estimators to variance reduction and
goal-oriented methods shall be the focus of the authors' future work.

The rest of this paper is organized as follows. In the next section,
we clarify certain details of the model and recall results concerning
the analysis of the finite element method. In \S
\ref{sec:frequency-analysis}, we present a formal frequency study of
the Galerkin error functional, which shows that the high frequency
content of the error, not observable on the computational grid, could,
at least in principle, be approximated by the computable low frequency
part under an assumption on a decay rate of Fourier modes. This
observation motivates a similar derivation of a computable error
estimate in \S \ref{sec:scales}, under a related assumption on
scales. In particular, in \S \ref{sec:comp-error-simple}, we give a
computable estimate of the pathwise observable of the Galerkin error
for the model problem \eqref{eq:model-problem} that relies on
computations on only one computational mesh; that is, the computations
are on one resolution level. We then present numerical experiments to
test the proposed estimator for a two-dimensional problem in \S
\ref{sec:2d-numer-exper}. Finally in \S \ref{sec:quadrature-error}, we
show that another assumption on scales yields an estimate of the
expected quadrature error committed in the finite element method.

%
\section{The Model}
\label{sec:model}
%

We consider the isotropic diffusion problem
\begin{equation}
  \label{eq:rpde}
  -\divergence (a(\omega,x) \nabla u(\omega,x)) = f(\omega,x)
\end{equation}
for $(\omega,x) \in \Omega \times \Gamma$ where $\Gamma$ is an bounded open subset of $\RR^d$, for $d \geq 1$, and $(\Omega, \mathbb{F}, \prob)$ is a given
probability space. To simplify our notation of the residual, we assume
the homogeneous Dirichlet boundary condition $u(x) = u(\omega,x) = 0$
for all $x \in \partial \Gamma$ $\as$ This equation arises in the
geophysics literature in the study of time-independent groundwater
flow on the local scale which is defined to be on the order of $10^2$
meters (\cite{Delhomme:1979, Dagan:1986, DaganNeuman:2005}). In this
setting, equation \eqref{eq:rpde} is Darcy's law with continuity where
$a$ represents the log hydraulic conductivity, $f$ is a given source
term, and the unknown $u$ represents the water pressure. A common
feature of groundwater flow on the local scale is the spatial
heterogeneity of the medium. In applications, prescribing $a$
precisely requires more information than is reasonably possible to
acquire. This uncertainty in the problem data is thus incorporated by
modeling $a$, and possibly $f$, as random fields. In applications to
subsurface flow, the law of $a$ is typically assumed to be lognormal
where the hydraulic conductivity, the Gaussian field $\log a$, has a
Lipschitz covariance. For example, in two dimensions, one could
consider the two-point isotropic covariance function given by $\cov
(x, y) = \sigma^2 e^{-|x - y|/\ell}$, where $\sigma^2$ is the
variance, $\ell$ the correlation length, and $|\cdot|$ is the
Euclidean norm (see \cite{HoeksemaKitanidis1985, Dagan:2005,
  DaganEtAl:2006}). Another feature of groundwater flow on the local
scale is that the correlation lengths involved are typically short;
that is, the length scale for $\ell$ is significantly smaller than the
scale of the problem domain but still too large for the application of
stochastic homogenization techniques
\cite{CliffeGilesScheichlTeckentrup:2011}.

We begin by introducing the objects of interest to our finite element
analysis. We study the variational form of equation \eqref{eq:rpde}
parameterized by $\omega \in \Omega$. That is, we seek $u \in
\mathcal{H}^{1}(\Gamma)$ such that
\begin{equation}
  \label{eq:variational-form}
  \int_\Gamma a(x) \nabla u(x) \cdot \nabla v(x) \dd x 
  = \int_\Gamma f(x) v(x) \dd x  \qquad \as
\end{equation}
for all $v \in V = \mathcal{H}_0^1(\Gamma) := \{ v \in
\mathcal{H}^{1}(\Gamma): v(x) = 0 \text{ for } x \in \partial \Gamma
\;\as \text{ in trace sense}\}$.  Note that this formulation seeks a
solution for almost all $\omega \in \Omega$. Let $V_h$ be the space of
continuous piecewise linear functions on a mesh in $\Gamma$ vanishing
along $\partial \Gamma$. The Galerkin formulation of
\eqref{eq:variational-form} is: find $u_h \in V_h$ such that
\begin{equation}
  \label{eq:galerkin-form}
  \int_\Gamma a(x) \nabla u_h(x) \cdot \nabla v_h(x) \dd x 
  = \int_\Gamma f(x) v_h(x) \dd x  \qquad \as
\end{equation}
for all $v_h \in V_h \subset V$. In the finite element approximation
of \eqref{eq:galerkin-form}, a further error is committed by replacing
the stiffness matrix components
\begin{equation*}
  \int_\Gamma a(x) \nabla\phi_i(x) \cdot \nabla\phi_j(x) \dd x
\end{equation*}
with quadrature based ones, which are of the form
\begin{equation*}
  \int_\Gamma \bar{a}(x) \nabla\phi_i(x) \cdot \nabla\phi_j(x) \dd x
\end{equation*}
so that the solution, including the contribution from the quadrature
error, is given by $\bar{u}_h \in V_h$ such that
\begin{equation}
  \label{eq:fem-form}
  \int_\Gamma \bar{a}(x) \nabla\bar{u}_h(x) \cdot \nabla v_h(x) \dd x 
  = \int_\Gamma \bar{f}(x) v_h(x) \dd x  \qquad \as
\end{equation}
for all $v_h \in V_h \subset V$. As posed,
\eqref{eq:variational-form}, \eqref{eq:galerkin-form} and
\eqref{eq:fem-form} are not uniformly elliptic with respect to
$\omega$. Nevertheless, conditions under which unique solutions exist
in the standard Bochner solution spaces,
$L^2(\Omega,\mathbb{F},\mathcal{H}_0^{1})$, can be found, for example,
in \cite{Charrier:2012,CharrierScheichlTeckentrup:2013}, where the
Fernique theorem \cite{Fernique:1975} is used to overcome the lack of
uniform ellipticity.

The analysis of finite element methods for elliptic PDE with
sufficiently regular and uniformly positive and bounded stochastic
coefficients is well established, see \cite{DebBabuskaOden:2001,
  BabuskaChatzipantelidis:2002, BabuskaTemponeZouraris:2004,
  MatthiesKeese:2005, FrauenfelderSchwabTodor:2005,
  BabuskaNobileTempone:2007, BarthSchwabZollinger:2011}.  These
analyses do not extend to model \eqref{eq:rpde} as the conductivity is
not guaranteed to be uniformly bounded and the regularity of the true
solution is low. Instead, the two predominant strategies for
approximating moments of observables are based on stochastic Galerkin
methods and MC finite element methods.

Stochastic Galerkin methods rely on a truncated Karhunen--Lo\`eve
expansion to obtain a reduced, or parametric, PDE that separates, in
some sense, the problem's stochastic and deterministic
components. Results in this direction have focused on questions of
well-posedness, convergence rates, and \emph{a priori} error estimates
\cite{Gittelson:2010, MuglerStarkloff:2011, Charrier:2012} and on
adaptive schemes using the energy norm \cite{Gittelson:2013,
  EigelEtAl:2013, EigelMerdon:2014}. Due to the short correlation
lengths involved in model \eqref{eq:rpde}, the reduced PDE typically
results in a very high-dimensional stochastic parameter space that can
be computationally limiting for calculating moments of observables.

An altogether different approach is to use MC methods.  Such
methods are favorable in high-dimensional situations since the
convergence rates are dimensionally independent. We mention
\cite{CliffeGilesScheichlTeckentrup:2011,
  CharrierScheichlTeckentrup:2013, TeckentrupEtAl:2013,
  GrahamScheichlUllmann:2013aa}, which analyze MLMC methods applied to
various finite element methods for elliptic PDE with lognormal
coefficients. In particular, \cite{GrahamScheichlUllmann:2013aa}
provides an analysis of mixed finite element methods that are useful
for considering quantities of interest that depend on
mass-conservative representation of the flux. All of these works focus
on \emph{a priori} estimates of the Galerkin finite element error.

In contrast, we give \emph{computable}, goal-oriented, duality-based
estimates for finite element methods in the MC framework. As discussed
in the introduction, the expected error in the observable for the
finite element approximation of \eqref{eq:rpde} is given by
\begin{equation}
  \label{eq:total-error}
  \mexp \left[ \left(g, u - \bar{u}_h\right) \right] =
  \mexp \left[ \left(g, u-u_h\right) \right] + \mexp \left[ \left(g,
      u_h - \bar{u}_h\right) \right] =: \mexp \mathcal{E}(g) + \mexp
  \mathcal{Q}(g),
\end{equation} where $\mexp \mathcal{E}(g)$ and $\mexp
\mathcal{Q}(g)$ are the expected Galerkin and quadrature errors,
respectively, for the given observable.  To be precise, we provide
computable estimates for $\mathcal{E}(g)$, and hence also for
$\mexp \mathcal{E}(g)$, and for $\mexp \mathcal{Q}(g)$ committed in a
standard piecewise linear finite element approximation of \eqref{eq:rpde}. 
For deterministic problems with smooth conductivities, Johnson and his
collaborators have long asserted the utility of such computable
duality-based estimators (for example, see \cite{ErikssonEtAl:1995}).

We shall first examine the problem posed in the introduction of
estimating $\mathcal{E}(g)$. Our estimates for error functionals of
this form rely on local error indicators computed using the finite
element approximation $u_h$ and the solution $\lambda$ to the
following dual problem: find $\lambda \in V$ such that
\begin{equation}\label{eq:weakdual}
  \int_\Gamma a(x) \nabla \lambda(x) \cdot \nabla v(x) \dd x 
  = \int_\Gamma g(x)v(x) \dd x  \qquad \as
\end{equation}
for all $v \in V$. A simple calculation, using integration by parts
and Galerkin orthogonality, shows that,
\begin{align*}
  \mathcal{E}(g) &= \left(u - u_h, - \nabla \cdot (a \nabla \lambda)\right)\\
  &= \left(a \nabla(u - u_h), \nabla \lambda\right)\\
  &= \left(a \nabla(u - u_h), \nabla (\lambda-v_h)\right)\\
  &= (f, \lambda-v_h) - (a \nabla u_h, \nabla (\lambda-v_h))\\
  &= \int_\Gamma \left(f + \nabla \cdot (a \nabla u_h)\right)
  (\lambda- v_h) \dd x
\end{align*}
for any $v_h \in V_h$, in the sense of distributions. This is a
similar calculation as in equation \eqref{u_h_residual_jump}. The error based
on the residual is therefore, in distribution form,
\begin{equation}
  \label{eq:error-density}
  \mathcal{E}(g) = \int_\Gamma R(\omega,x) (\lambda (\omega,x) - v_h(x)) \dd x  =\langle R,\lambda-v_h\rangle,
\end{equation}
for any $v_h \in V_h$, where the residual $R := f + \nabla \cdot (a
\nabla u_h)$ is a distribution containing also the jump terms at the
edges since the residual is a measure on the edges and not a
function. For problem \eqref{eq:model-problem}, this quantity reduces
to
\begin{equation}
  \label{eq:error-density-model-1D}
  \mathcal{E}(g) = \int_0^1 (a u_h^\prime)^\prime (\lambda - v_h) \dd x + \left(1-u_h(1)\right)\left(\lambda(1)-v_h(1)\right),
\end{equation}
as shown in \eqref{u_h_residual_jump}, and the remaining part of the
residual, that does not vanish, is $R = a'u_h'$. A formal analysis of
the frequency content of \eqref{eq:error-density-model-1D} offers
further insight into suitable local error indicators.

%
\section{Numerical Study of the Frequency Content}
\label{sec:frequency-analysis}
%

Next, we undertake a formal study of the frequency content of the
residual error \eqref{eq:error-density-model-1D} for the simple model
problem \eqref{eq:model-problem}. The error contains high-frequency
content that is non-negligible. Under assumptions on the decay of the
frequency content, we observe that the high-frequency contribution in
the error observable can be estimated in terms of the low-frequency
content. Although computations are provided only for a one-dimensional
example, this frequency study provides insight into why the standard
\emph{a posteriori} analysis fails. Motivated by the observations in
this section, we derive error estimators in the subsequent sections
through an analysis of scales and give a computable estimator for
\eqref{eq:model-problem} requiring only one level of the mesh size in
\S \ref{sec:comp-error-simple}.

From equation \eqref{eq:error-density-model-1D}, we express the
remaining part of the residual and the dual terms in Fourier series,
\begin{equation*}
  R(x) = \sum_{n = - \infty}^{\infty} r_n e^{2\pi inx} \qquad\text{and}\qquad
  (\lambda - v_h)(x) = \sum_{n=-\infty}^{\infty} \lambda_n e^{2\pi inx},
\end{equation*}
where $r_n$ and $\lambda_n$ denote the Fourier coefficients and where
we are free to choose $v_h \in V_h$ to ensure a continuous periodic
extension of $\lambda - v_h$. In particular, we choose $v_h = \pi_h
\lambda$, where $\pi_h :V \to V_h$ is the nodal interpolant. Naively
choosing $v_h = 0$ would result in a decay rate for $\lambda_n$ of at
most $O(n^{-1})$, which is non-optimal due to the jump discontinuity
introduced by the particular choice of boundary conditions for the
model problem. Further, we shall split $\lambda - v_h$ into low- and
high-frequency components,
\begin{equation}
  \label{eq:splitting}
  \lambda-v_h = \sum_{|n|< n^\star} \lambda_{n} e^{2\pi inx} 
  + \sum_{|n| \geq n^\star} \lambda_n e^{2\pi inx}
  =: \underline{\lambda} + \overline{\lambda},
\end{equation}
for $n^\star = C/h$, where $C = C(\omega)$ is a constant independent
of $h$. The mesh size, $h$, limits the frequencies obtained in the
computed finite element solution, $u_h$. In this sense,
$\underline{\lambda}$ and $\overline{\lambda}$, respectively,
represent the low- and high-frequency components relative to the
residual that is computed using the finite element solution,
$u_h$. $\overline{\lambda}$ cannot be approximated based on $V_h$ or
$V_{\frac{h}{2}}$ but $\underline{\lambda}$ can be approximated based on $V_h$
and $V_{\frac{h}{2}}$.

In Figure~\ref{fig:comp-low-freq-rough-smooth}, an estimate of
\begin{equation*}
  |\mathcal{E}(1)| := | \sum_{n=-\infty}^{\infty} r_n \lambda_n |,
\end{equation*}
the Galerkin error, and
\begin{equation*}
  |\mathcal{E}_L(1)| := | \sum_{|n| < n^\star} r_n \lambda_n |,
\end{equation*}
the low-frequency component of the Galerkin error, for the observable
$g=1$, are plotted for a smooth conductivity, $a(x) = x+1$, and for
one realization of a rough lognormal conductivity, $a(x) = e^{B(x)}$,
where the field $a$ has covariance function,
\begin{equation*}
  \cov(a(x),a(y)) =  e^{\frac{1}{2}(x(1-x) + y(1-y))}(e^{\min(x,y) - xy} - 1).
\end{equation*}
Here the Fourier coefficients are given by the discrete Fourier
transform, $v_h = \pi_h \lambda$ is chosen, $\lambda$ is computed on a
reference mesh of size $h = 2^{-16}$, and all the finite element
calculations use a quadrature mesh of size $k = 2^{-23}$ to avoid
polluting the Galerkin error with quadrature error. We observe that in
the case of the smooth $a$, the contribution to the error from the
high-frequency component is expected to be negligible, as is indicated
in the left-hand panel. In the case of the rough lognormal $a$, see
the right-hand panel, the low-frequency component does not
sufficiently capture the Galerkin error indicating a non-negligible
contribution from $\overline{\lambda}$. Although the high frequency
component, which requires a space much finer than $V_h$ to be
resolved, contributes to the error in this setting, we are able to
estimate the contribution from $\overline{\lambda}$ in terms of
$\underline{\lambda}$ if an assumption on the decay of the Fourier
coefficients is made.

\begin{figure}[]
  \centering \subfloat{%
    \includegraphics[width=0.49\textwidth]{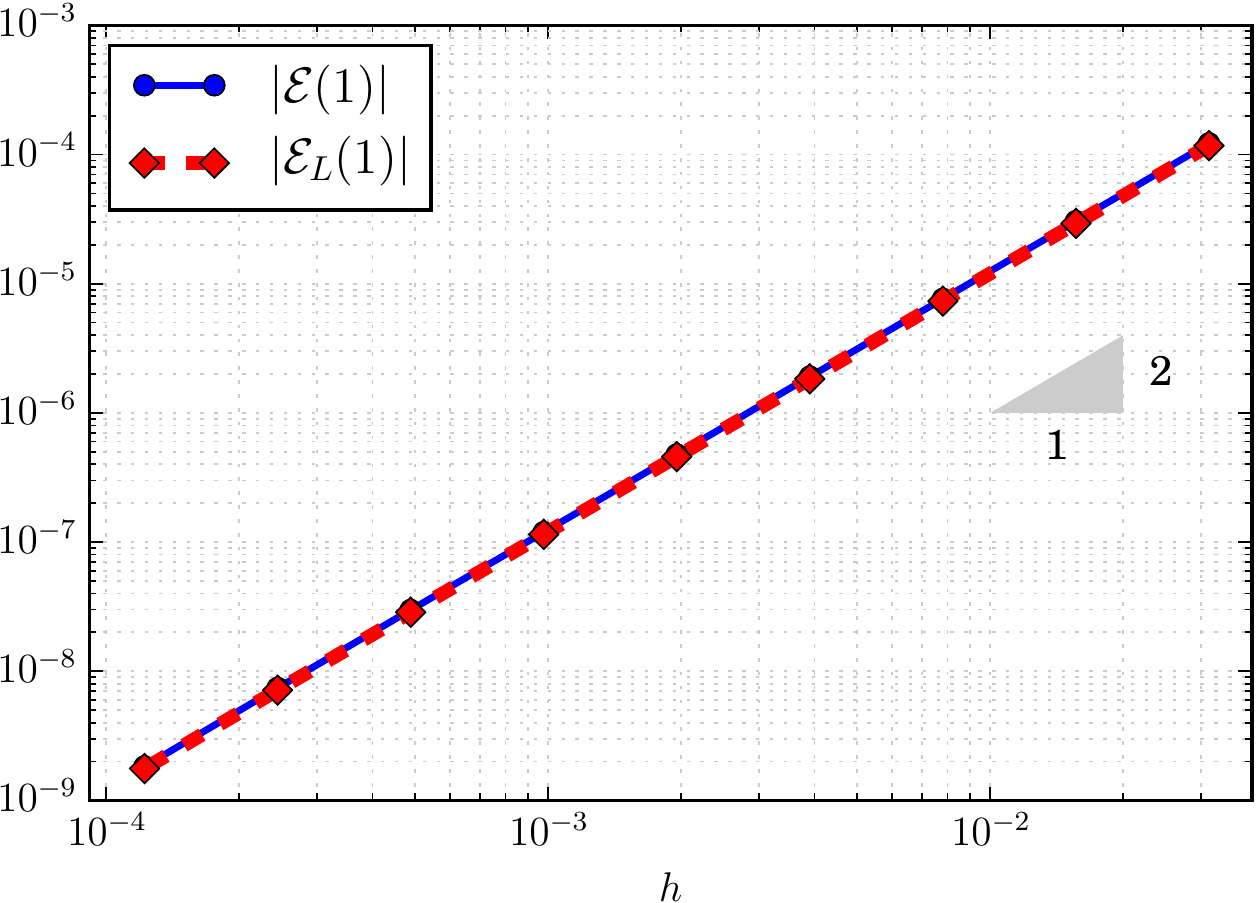}}
  \hfill \subfloat{%
    \includegraphics[width=0.49\textwidth]{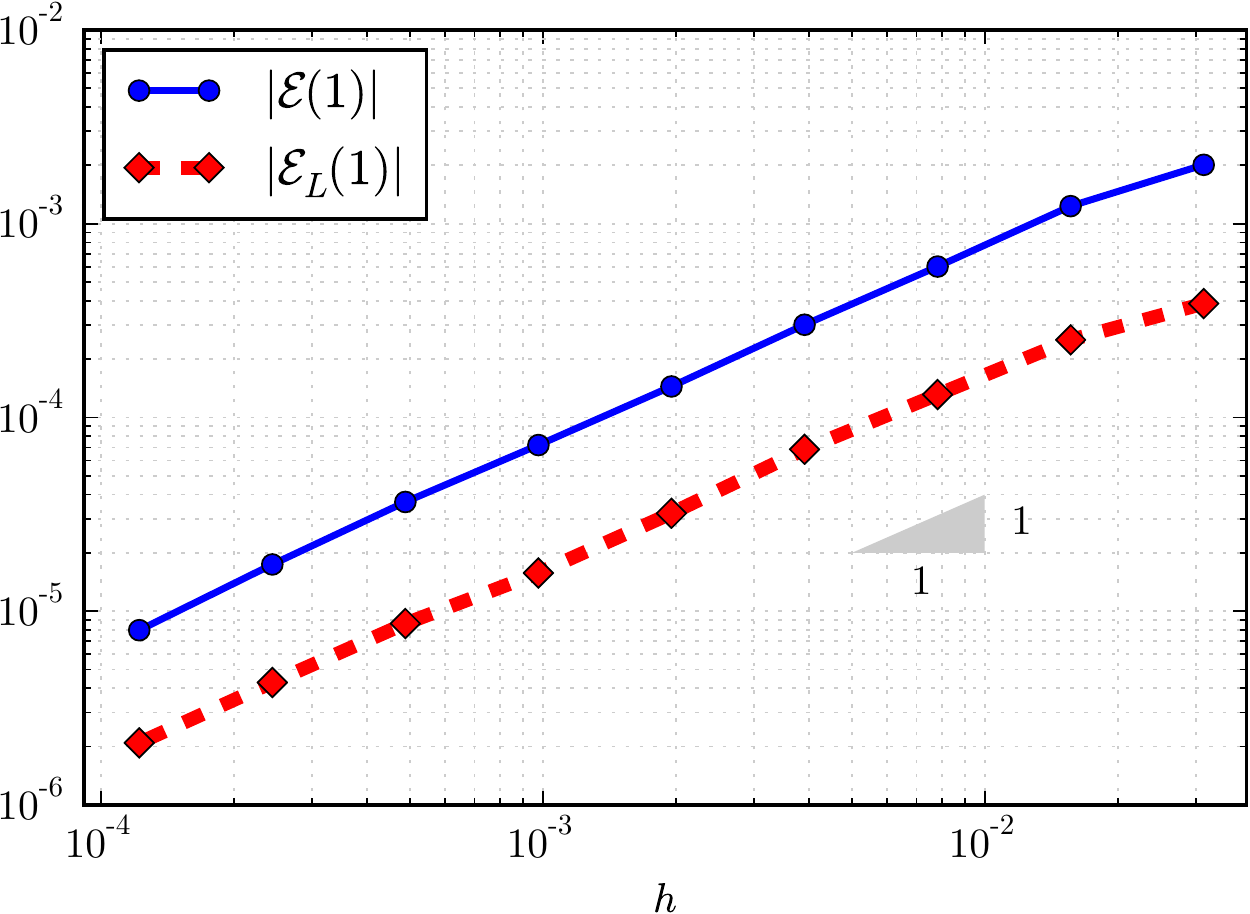}}%
  \caption{The high-frequency content contributes to the Galerkin
    error for one realization of the rough lognormal conductivity,
    $a(x) = e^{B(x)}$, (right-hand panel) in contrast to a smooth
    conductivity, $a(x) = 1+x$, (left-hand panel) where the
    contribution is negligible.}
  \label{fig:comp-low-freq-rough-smooth}
\end{figure}

\begin{remark}\rm
  Figure \ref{fig:comp-low-freq-rough-smooth} also demonstrates that
  the Galerkin error in the piecewise linear finite element
  approximation of \eqref{eq:model-problem} is on the order of $O(h)$
  in the case of the rough lognormal conductivity as compared to
  $O(h^2)$ for the smooth conductivity.
\end{remark}

If we assume that
\begin{equation}
  \label{eq:decay-assumption}
  | r_n \lambda_n | + | r_{-n}\lambda_{-n}| \leq C n^{-2\alpha},
\end{equation}
for positive $n$ and a positive $C = C(\omega)$, then
\begin{equation*}
  \mathcal{E}(g) \leq \tilde{C} h^{2\alpha-1}
\end{equation*}
for $\alpha \in (1/2, 3/2)$ for a random variable $\tilde{C} =
\tilde{C}(\omega)$ not depending on $h$. This estimate is obtained by
first applying the splitting \eqref{eq:splitting} to
\eqref{eq:error-density-model-1D} and then by using the decay
assumption \eqref{eq:decay-assumption}.

To wit, the frequency splitting yields
\begin{equation*}
  \mathcal{E}(g) = \int_{0}^{1} R(x) 
  (\underline{\lambda}-\pi_h \underline{\lambda})(x) \dd x 
  + \int_{0}^{1} R(x) \overline{\lambda}(x) \dd x.
\end{equation*}
Applying $\underline{\lambda} - \pi_h \underline{\lambda} \simeq C_0
h^2 \underline{\lambda}^{\prime\prime}$, a standard localization
estimate with deterministic interpolation constant $C_0$, the error
can then be approximated by
\begin{equation*}
  |\mathcal{E}(g)| \leq \sum_{|n| < n^\star} 
  C_0 h^2 n^2 |r_n| |\lambda_n| 
  + \sum_{|n| \geq n^\star} |r_n \lambda_n|.
\end{equation*} 
Then, we have
\begin{equation*}
  |\mathcal{E}(g)| \leq \sum_{|n| < C/h} 
  C_0 h^2 \frac{n^2}{n^{2\alpha}} 
  + \sum_{|n| \geq C/h} \frac{C}{n^{2\alpha}},
\end{equation*}
by the decay assumption \eqref{eq:decay-assumption}. The
high-frequency contribution, arising from the latter sum, can be
estimated by
\begin{equation*}
  2 \int_{C/h}^\infty s^{-2\alpha} \dd s =
  2 \left[\frac{-s^{-2\alpha +1}}{1-2\alpha}\right]_{C/h}^\infty
  = \frac{2}{2\alpha-1} \left(\frac{h}{C}\right)^{2\alpha-1},
\end{equation*}
provided $\alpha > 1/2$, and the low-frequency contribution, arising
from the first sum, by
\begin{equation*}
  2 C_0 h^2 \int_0^{C/h} s^{2-2\alpha} \dd s
  = \frac{2C_0}{3-2\alpha} C^{3-2\alpha}h^{2\alpha-1},
\end{equation*}
provided $\alpha < 3/2$.

Thus, we have the estimate
\begin{equation*}
  |\mathcal{E}(g)| \leq 2 \left(\frac{C_0 C^{3-2\alpha}}{3-2\alpha} +
    \frac{C^{1-2\alpha}}{2\alpha-1}\right) h^{2\alpha-1},
\end{equation*} 
for $\alpha \in (1/2, 3/2)$.  In contrast, for a smooth conductivity
$a$, we have
\begin{equation*}
  |\mathcal{E}(g)| \leq 2 \left( \frac{C_0}{2\alpha - 3} h^2 + \frac{1}{2\alpha-1} 
    h^{2\alpha-1} \right)
\end{equation*}
for $\alpha > 3/2$, where the contribution from the high frequencies
is clearly negligible. While the high frequency error contribution for
the rough lognormal conductivity cannot be observed on the
computational grid, the foregoing argument shows that it is of the
same order, $h^{2\alpha-1}$, as the low frequency error that is
possible to detect in a numerical simulation. Therefore, a naive
approach that only takes into account the smooth part of the error
would always be a factor wrong.

Returning to consider the validity of our initial assumption, we see
that \eqref{eq:decay-assumption} is satisfied for a constant $\alpha$
in the case of the simple model problem \eqref{eq:model-problem}.
Recall that the residual has the form $R(x) = a^\prime u_h^\prime =
B^\prime e^{B}u_h^\prime \approx W^\prime e^{B} u_h^\prime$, where
$W^\prime$ is a white noise in space. The white noise has a power
spectrum density that is flat, that is, $|\mathcal{F}_n(W^\prime)|^2 =
O_P(1)$, that is, $O(1)$ $\as$, where $\mathcal{F}_n(f)$ denotes the
$n$th Fourier mode of $f(x)$. Hence $|\mathcal{F}_n(R)| = O_P(1)$. For
the dual, we have the explicit representation $\lambda^\prime = G/a$,
where $G$ is the primitive of $g$. If we suppose $G=1$ and approximate
$1/a \approx 1 - B \approx 1-W$, then
\begin{equation*}
    |\mathcal{F}_n(\lambda)| = O_P(n^{-1}) |\mathcal{F}_n(\lambda^\prime)| 
    = O_P(n^{-1}) |\mathcal{F}_n(W)| = O_P(n^{-2}).
\end{equation*}
Assuming more generally that $g \in C^0([0,1])$, we have by Parseval's
identity that $\| g\|_2 = \sum_k |\mathcal{F}_k(g)|^2 < \infty$, so we
may assume $\mathcal{F}_n(g) \to 0$ as $n \to \infty$. Then we obtain
the same rate,
\begin{align*}
  |\mathcal{F}_n(\lambda)|
  &= O_P(n^{-1}) |\mathcal{F}_n(\lambda^\prime)|\\
  &= O_P(n^{-1}) |\mathcal{F}_n(G(1-W))| \\
  &= O_P(n^{-2}) \left(|\mathcal{F}_n(g)| + |\mathcal{F}_n(gW +
    GW')|\right) = O_P(n^{-2}).
\end{align*}
Here we have used the convolution rule for Fourier series, Parseval's
identity, and H\"older's inequality to conclude that
\begin{equation*}
  \mathcal{F}_n(gW) 
  = (\mathcal{F}(g)*\mathcal{F}(W))_n 
  = \sum_k \mathcal{F}_k(g) \mathcal{F}_{n-k}(W) \leq  \| g \|_2 \|W\|_2 
  = O_P(1),
\end{equation*}
since $W$ is in $C^0([0,1])$ and bounded, and thus $W \in
  L^2(0,1)$ $\as$ Computing the integral of $\mathcal{F}_n(GW')$ is
  slightly more cumbersome, but once again one can use that $\|W\|_2 =
  O_P(1)$, neglecting constants, write the integral with $W'$ expanded
  in the Karhunen-Lo\`{e}ve basis to see that
\begin{align*}
  \mathcal{F}_n(GW')
  &= \sum_k \mathcal{F}_{n+k}(G) \mathcal{F}_k(W^\prime)\\
  &= \sum_k \frac{\mathcal{F}_{n+k}(g)}{|n+k| +1}  \mathcal{F}_k(W^\prime)\\
  &= \sum_k \mathcal{F}_{n+k}(g) \mathcal{F}_k\left(\frac{W^\prime}{1+|k|}\right)\\
  &= \sum_k \mathcal{F}_{n+k}(g) \mathcal{F}_k(W )\\
  &\leq \| g\|_2 \|W\|_2 = O_P(1).
\end{align*}
Thus, the decay rate is satisfied with $\alpha = 1$ in the case of the
simple model problem when $g\in C^{0}([0,1])$.

This rate is also observed experimentally for $g=1$ as well as $g =
\delta_{\frac{1}{2}}$. Using the explicit representations available in
the one-dimensional model problem, the product of the Fourier
coefficients corresponding to $\mathcal{E}(1)$, and also for
$\mathcal{E}(\delta_{\frac{1}{2}})$, are plotted in
Figure~\ref{fig:decay-test}, where $\delta_{\frac{1}{2}}(x) :=
\delta(x - 1/2)$ is the $\delta$-function centered at $x = 1/2$. Here
the computations are for one sample path of the log-Brownian bridge
process, $a = e^B$, for elements of size $h = 2^{-10}$, with a
quadrature mesh of size $k = 2^{-25}$, and the decay rate is fitted
using the method of least squares. The Fourier content decays at a
constant rate, corresponding to $\alpha \approx 1$, for both the
observables $g = 1$ and $g = \delta_{\frac{1}{2}}$.  In particular,
for constant $\alpha = 1$, the contributions from the high- and
low-frequency components are $2C^{-1}$ and $2C_0 C$, respectively, and
therefore
\begin{equation*}
  \frac{\text{low}}{\text{total}}=  \frac{\text{low}}{\text{low}+\text{high}} 
  = \frac{1}{1+\frac{1}{C_0C^2}}
\end{equation*}
is the ratio of the contribution of the low-frequency part to the
total frequency content of the error.

\begin{figure}[]
  \centering \subfloat{%
    \includegraphics[width=0.49\textwidth]{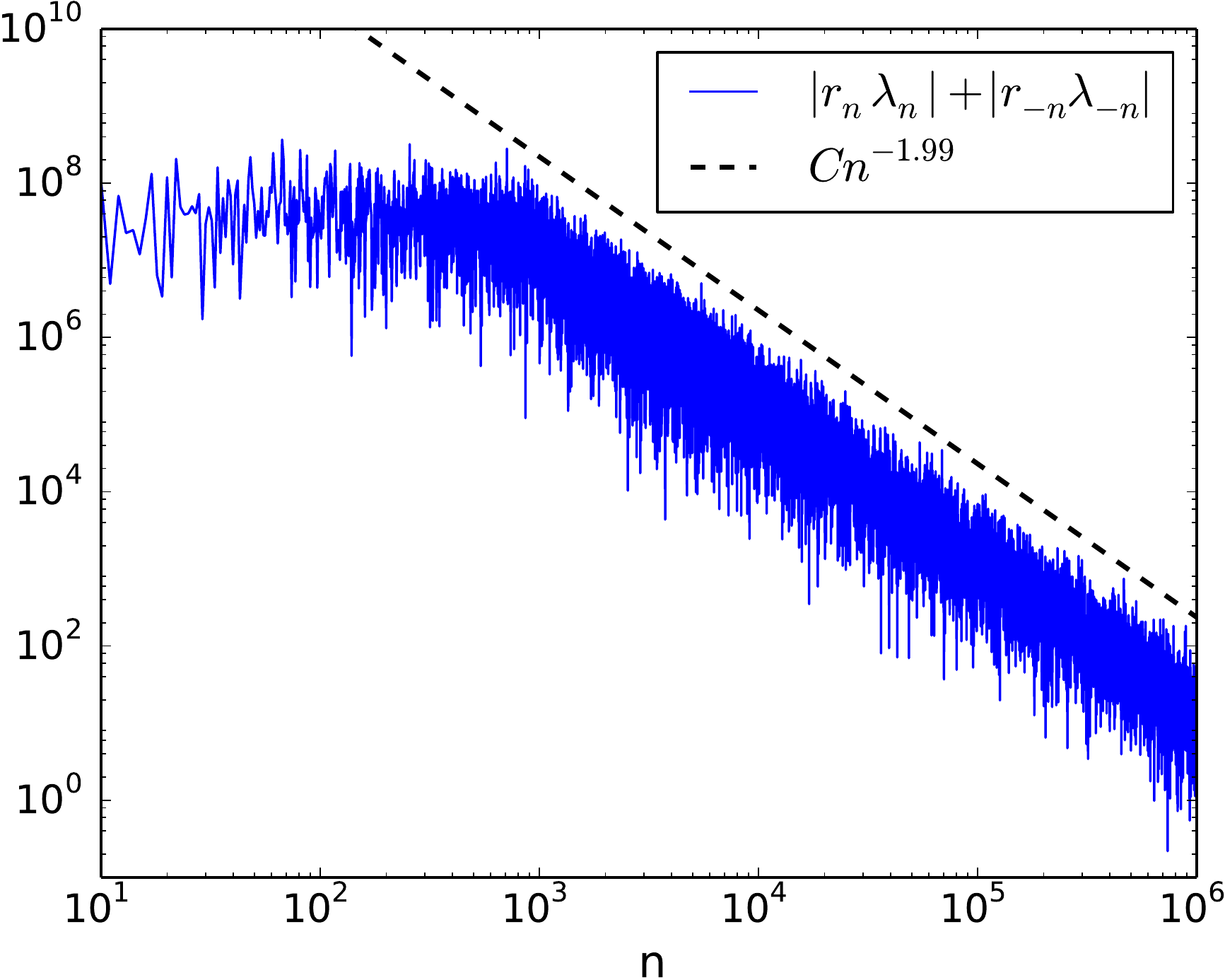}
  }\hfill \subfloat{%
    \includegraphics[width=0.49\textwidth]{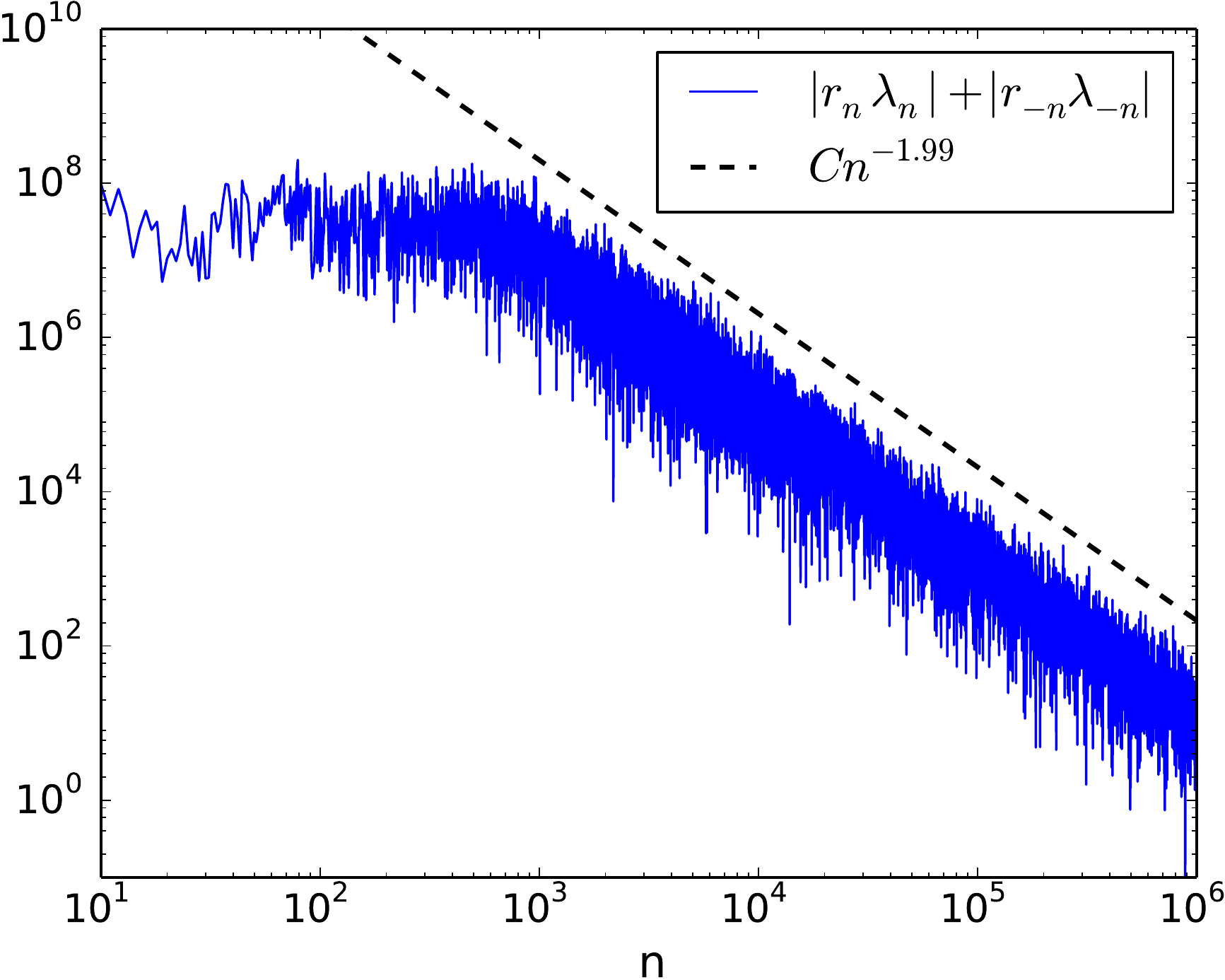}
  }%
  \caption{Pathwise, the Fourier content decays at a constant rate,
    observed to be $O(n^{-1.99})$, for the observables $g = 1$
    (left-hand panel) and $g = \delta_{\frac{1}{2}}$ (right-hand
    panel).}
  \label{fig:decay-test}
\end{figure}

With the insight obtained from the formal analysis in the present
section, we anticipate that only a fraction of the total error can be
found on the computational grid used. If this fraction would be known,
then the total error could be estimated as an error estimator from the
computational grid divided by that fraction. In the next section we
show rigorously that frequency splitting can be achieved under an
assumption on scales related to \eqref{eq:decay-assumption}. The
analysis given there also gives the factor that should multiply the
``observable'' error contribution from a computational grid.

%
\section{Analysis of scales}
\label{sec:scales}
%

\subsection{Frequency splitting via an assumption on scales}
\label{sec:freq-split-via}

Presently we derive the error estimators via a general telescoping
argument in which the frequency splitting required to drive our
analysis is achieved by an assumption on scales. This analysis yields
a similar estimate as the frequency study in the previous section for
the one-dimensional model problem and can be readily applied to
higher-dimensional problems.  We will need the solutions $\lambda_h$
to the discretized version of the dual equation \eqref{eq:weakdual},
given by
\begin{equation}\label{eq:discretedual}
  \int_\Gamma a(x) \nabla \lambda_h(x) \cdot \nabla v_h(x) \dd x 
  = \int_\Gamma g(x)v_h(x) \dd x  \qquad \as
\end{equation}
for all $v_h \in V_h \subset V$.
\begin{theorem}\label{thm:GalerkinScales}
  Assume that, for positive constants $\alpha$ and $\beta$, and for a
  positive random variable $C=C(\omega)$, the discrete solutions $u_h$
  and $\lambda_h$ to the equations \eqref{eq:galerkin-form} and
  \eqref{eq:discretedual} satisfy
  \begin{equation}
    \label{eq:freq-split-by-scales}
    \int_\Gamma a(x) \left|\nabla (u_{h/2}-u_h)(x) 
      \cdot \nabla (\lambda_{H/2}-\lambda_H )(x)\right| \dd x 
    \leq C h^\alpha H^\beta
  \end{equation}
  for sufficiently small mesh sizes $h$ and $H$. Then the Galerkin
  error $\mathcal{E}(g)$, defined in \eqref{eq:total-error}, satisfies
  \begin{equation*}
    \label{eq:upper-bound-galerkin-error}
    |\mathcal{E}(g)| \leq \frac{C h^{\alpha+\beta}}{(1-2^{-\alpha})(1-2^{-\beta})}.
  \end{equation*} 
\end{theorem}

{\em Proof}. Letting $v_h = \lambda_h$, we rewrite the Galerkin error
\eqref{eq:error-density} as
\begin{equation*}
  \mathcal{E}(g) = \int_\Gamma a(x)  \nabla (u - u_h)(x)
  \cdot \nabla (\lambda - \lambda_h)(x) \dd x.
\end{equation*}
Expanding the quantities containing differences in the primal and dual
variables in telescoping sums in $m$ and $n$, we obtain
\begin{align*}
  |\mathcal{E}(g)| & = \left| \sum_{m,n=0}^\infty \int_\Gamma a(x)
                     \nabla (u_{h/2^{m+1}}-u_{h/2^{m}})(x) \cdot \nabla
                     (\lambda_{h/2^{n+1}}-\lambda_{h/2^{n}})(x) \dd x \right| \\
                   &\leq \sum_{m,n = 0}^\infty \int_\Gamma a(x) \left| \nabla
                     (u_{h/2^{m+1}}-u_{h/2^{m}})(x) \cdot \nabla
                     (\lambda_{h/2^{n+1}}-\lambda_{h/2^{n}})(x)\right|
                     \dd x\\
                   & \leq \sum_{m,n =0}^\infty C h^\alpha h^\beta
                     2^{-m\alpha}2^{-n\beta} =
                     \frac{Ch^{\alpha+\beta}}{(1-2^{-\alpha})(1-2^{-\beta})}. \qquad\endproof
\end{align*}

One can estimate $\mathcal{E}(g)$ by computing
\begin{equation*}
  \tilde{F}(h) := \int_\Gamma a \nabla (u_{h/2}-u_h)
  \cdot \nabla(\lambda_{h/2}-\lambda_h) \dd x \approx C^\prime(\omega) h^{\alpha+\beta},
\end{equation*}
a quantity of the same order as $\mathcal{E}(g)$, for a random
variable $C^\prime$ independent of $h$, which relies on local error
indicators on two levels, that is, on two nested computational
meshes. In practice, we shall estimate $\mathcal{E}(g)$ by
\begin{equation}\label{eq:Fdef}
  F(h) := \int_\Gamma a(x) \left|\nabla (u_{h/2}-u_h)(x) 
    \cdot \nabla(\lambda_{h/2}-\lambda_h)(x)\right| \dd x,
\end{equation}
for the reasons outlined below.

For some observables, the Galerkin error may be small due to
non-stochastic cancellations in the error indicators. In
Figures~\ref{fig:indicators} and \ref{fig:indicators-pathalogical}, we
plot, element-wise, the local error indicators for several observables
for the one-dimensional model problem to illustrate the possible
effects of cancellations. That is, we split the domain of the integral
in $\tilde{F}(h)$ over each $h$-element and plot the resulting
element-wise contributions to the Galerkin error corresponding to
$g = 1$ and $g = \delta_{\frac{1}{2}}$ in Figure~\ref{fig:indicators}
and $g = \cos(2\pi \cdot)$ in
Figure~\ref{fig:indicators-pathalogical}. These error indicators are
computed for $h = 2^{-9}$ using a quadrature mesh, $k = 2^{-22}$.  In
contrast to the indicators corresponding to $g = \cos(2\pi \cdot)$,
displayed in Figure~\ref{fig:indicators-pathalogical}, we observe that
the error indicators corresponding to $g = 1$ and
$g = \delta_{\frac{1}{2}}$ will not exhibit extensive cancellations
when the error estimator is computed.

\begin{figure}[]
  \centering \subfloat{%
    \includegraphics[width=0.49\textwidth]{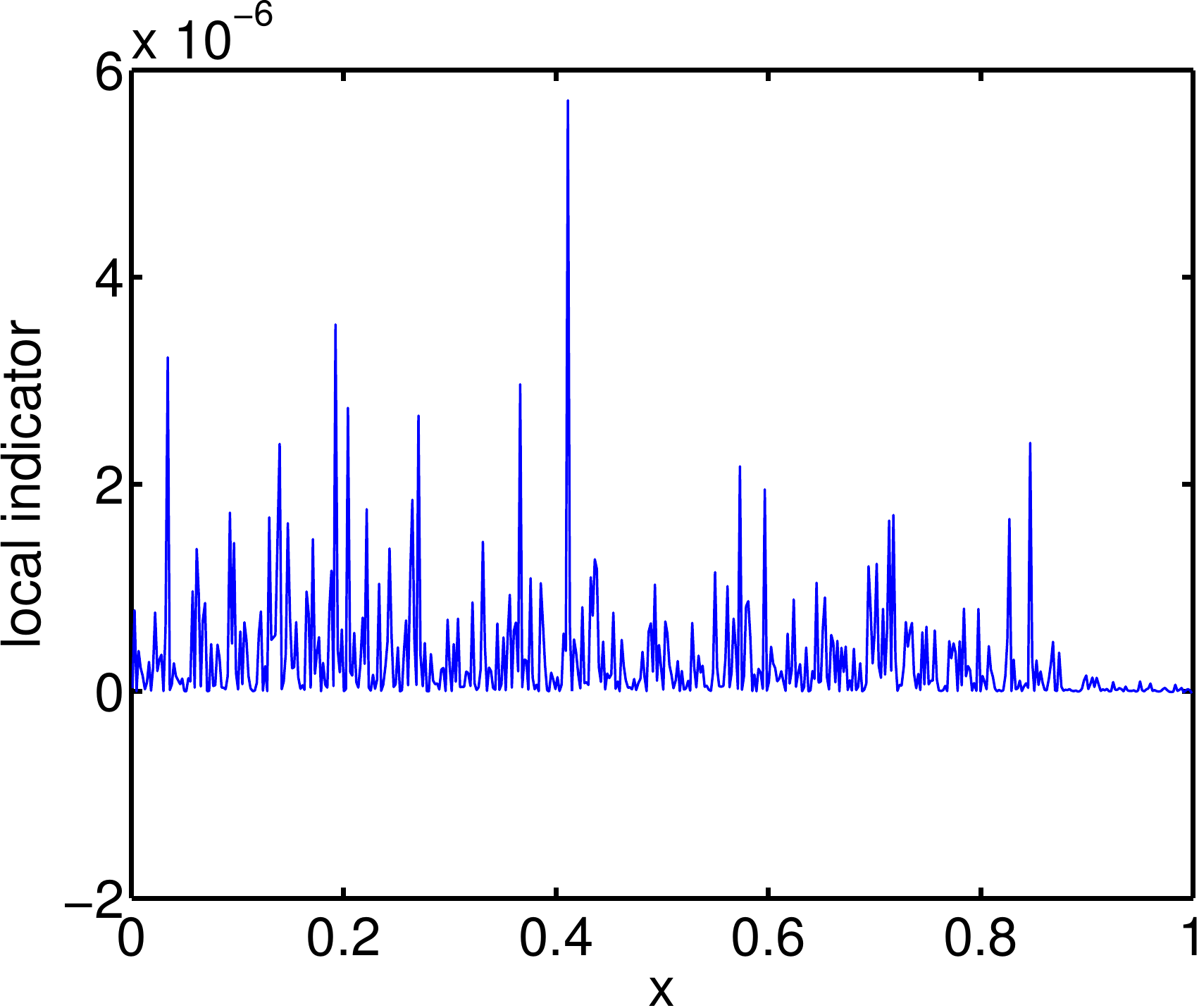}
  }\hfill \subfloat{%
    \includegraphics[width=0.49\textwidth]{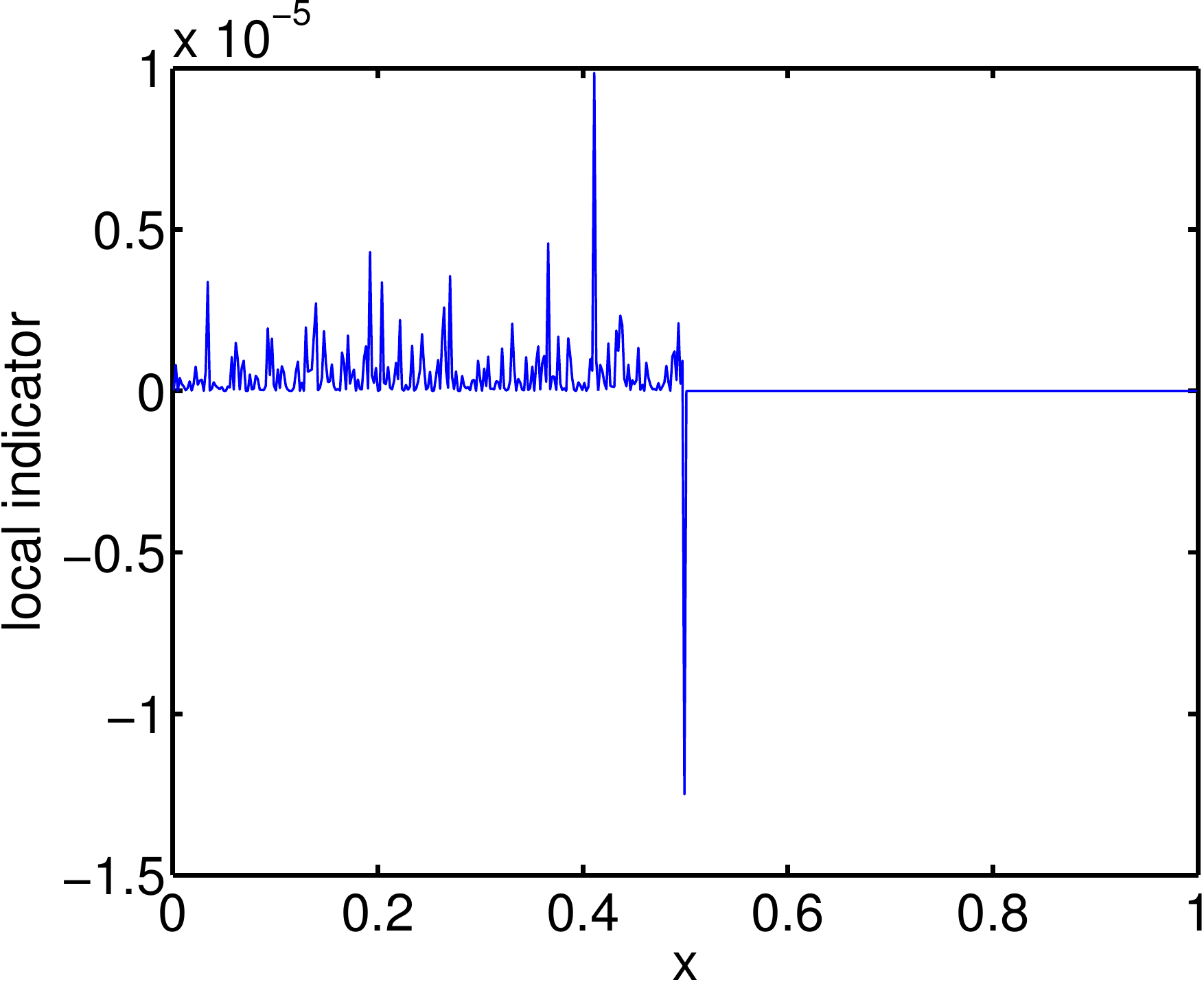}
  }
  \caption{The local error indicators for one realization of the error
    committed in approximating two generic observables, $g=1$
    (left-hand panel) and $g = \delta_{\frac{1}{2}}$ (right-hand
    panel), do not sum to zero.}
  \label{fig:indicators}
\end{figure}
\begin{figure}[]
  \centering \subfloat{%
    \includegraphics[width=0.49\textwidth]{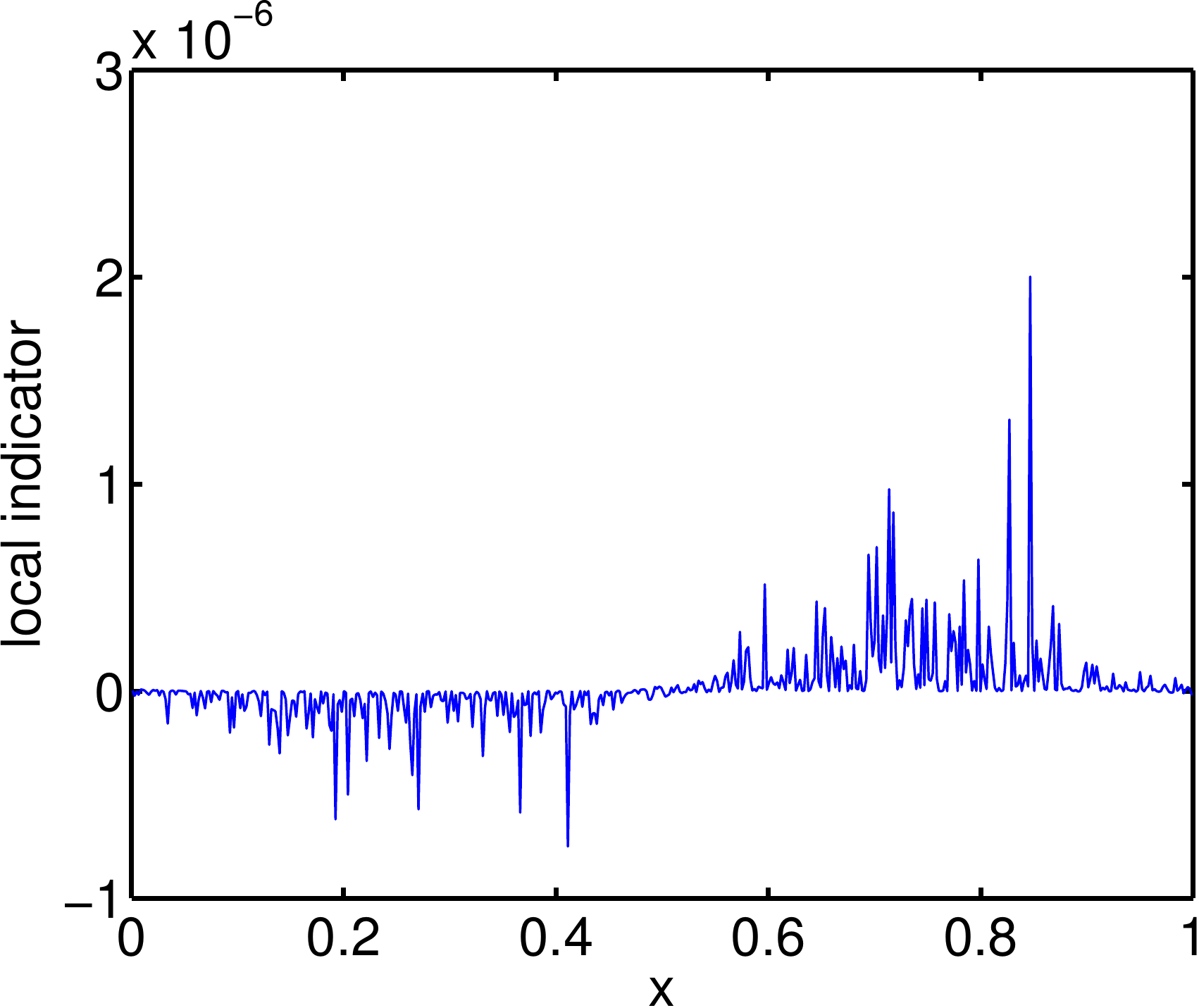}
  }%
  \caption{The local error indicators for one realization of the error
    committed in approximating a pathological observable,
    $g = \cos(2\pi \cdot)$, specially constructed to exhibit extensive
    cancellations in the estimator.}
  \label{fig:indicators-pathalogical}
\end{figure}

Trivially, the estimator may vanish due to cancellations of the
indicators, for a particular $h$ or for a particular choice of
observable. However, it is not easy to say when such cancellations
occur, or how to take advantage of them in order to make better
predictions of the size of the error.  Furthermore, if we require that
the signed estimator $F(h)$ is bounded away from zero, we obtain the
following corollary to Theorem \ref{thm:GalerkinScales}:

\begin{corollary}\label{cor:scales}
  If the estimator $F(h)$, defined in \eqref{eq:Fdef}, satisfies
  \begin{equation}
    \label{eq:estimator-bounded}
    c \leq \frac{F(h)}{h^{\alpha+\beta}} \leq C
  \end{equation}
  for positive constants $c = c(\omega)$ and $C = C(\omega)$, both
  independent of $h$, the Galerkin error $\mathcal{E}(g)$, defined in
  \eqref{eq:total-error}, satisfies
  \begin{equation*}
    |\mathcal{E}(g)|
    \leq  \frac{C F(h)}{c (1-2^{-\alpha})(1-2^{-\beta})}.
  \end{equation*}
\end{corollary}

Next, we derive a computable estimate for the one-dimensional model
problem equation \eqref{eq:model-problem}. Then, in the remainder of
this section, we determine the factor in the above estimate and
provide some experiments testing our estimator for the model problem.

\subsection{Computable error for the simple model problem}
\label{sec:comp-error-simple}

In Theorem \ref{thm:errorrepresentation} below, we give a computable
estimator, based on local error indicators, for the simple model
problem \eqref{eq:model-problem}. This estimator relies on a
representation of the discrete primal and dual solutions, given in
Lemma \ref{lem:discretesolutions}, where the value of $a$ in the
variational formulation is averaged over the $h$-elements. That is,
let $a_h$ be the spatial average of $a$ over the $h$-elements,
\begin{equation}
  \label{eq:spatial-avg-a}
  a_h(x) = \sum_{K_h} \ind{K_{h}} (x) \frac{1}{|K_h|} \int_{K_h} a(s) \dd s,
\end{equation}
for all $h$-elements $K_h$, where
\begin{equation*}
  \ind{K_{h}}(x) = \begin{cases} 1 & x \in K_h\\ 0 & x \not\in K_h \end{cases}
\end{equation*}
and where $|K_h| = \int_{K_h}\!\!\dd x$ denotes the size of element
$K_h$. The finite element solution $u_h$ to \eqref{eq:model-problem}
satisfies the discrete variational equation
\begin{equation}
  \label{eq:uhvar}
  \int_0^1 a u_h' v_h' \dd x  = v_h(1)   \qquad \as
\end{equation}
for all $v_h \in V_h$, where $V_h$ is the space of all piecewise
linear functions $v_h$ on $(0,1)$ with grid spacing $h$ that satisfy
$v_h(0)=0$.  Since $u_h^\prime$ and $v_h^\prime$ are constant on each
$h$-element we have
\begin{equation}
  \label{eq:uhdef}
  \int_0^1 a u_h' v_h' \dd x = \int_0^1 a_h u_h' v_h' \dd x. 
\end{equation}
Likewise, the discrete dual, $\lambda_h$, is the solution of the
variational equation
\begin{equation}
  \label{eq:lambdahdef}
  \int_0^1 a \lambda_h' v_h' \dd x = \int_0^1 a_h \lambda_h' v_h' \dd x =
  \int_0^1 g v_h \dd x,
\end{equation}
for all piecewise linear test functions, $v_h$. For the observable
$g$, let
\begin{equation}
  \label{eq:Gdef}
  G(x)=-\int_x^1g(s) \dd s
\end{equation}
denote the primitive function of the observable $g$. We then have the
following representations for the derivatives of the discrete primal
and dual solutions.

\begin{lemma}
  \label{lem:discretesolutions}
  The derivatives of the discrete primal and dual solutions $u_h$ and
  $\lambda_h$ are given by
  \begin{equation}\label{eq:discreteDerivatives}
    \begin{split}
      u_h'&=1/a_h,\\
      \lambda_h'&=G_h/a_h,
    \end{split}
  \end{equation}
  where
  \begin{equation*}
    G_h(x) = 
    \sum_{K_h} \ind{K_h}(x) \frac{1}{|K_h|} \int_{K_h} G(s) \dd s.
  \end{equation*}
  
\end{lemma}
\begin{proof}
  The expression for $u_h$ follows directly from \eqref{eq:uhvar} and
  \eqref{eq:uhdef}.  The expression for $\lambda_h$ follows by
  integration by parts of the right-hand side in
  \eqref{eq:lambdahdef}:
  \begin{equation*}
    \int_0^1 a_h \lambda_h' v_h' \dd x = \int_0^1 G v_h' \dd x = \int_0^1 G_h v_h' \dd x,
  \end{equation*} 
  where we use $v_h(0)=G(1)=0$.  Since this equation must hold for all
  test functions, $v_h$, the expression for $\lambda_h$ in the lemma
  follows.
\end{proof}

Each interval in the grid with step length $h$ corresponds to
precisely two subintervals in the grid with length $h/2$. In order to
have concise notation, entities corresponding to the left of these two
intervals are marked with a minus ($-$) sign, and entities
corresponding to the right subinterval are marked with a plus ($+$)
sign. The following somewhat surprisingly simple equality gives a
computable estimator for the observable for the Galerkin error for the
simple model problem \eqref{eq:model-problem} based on only one mesh
size.

\begin{theorem}\label{thm:errorrepresentation}
  \begin{equation*}
    \tilde{F}(h) = \int_0^1 a(u_{h/2}-u_h)^\prime
    (\lambda_{h/2}-\lambda_h)^\prime \dd x
    =\sum_{K_h} \frac{h^3}{16}a_h^* D^2 u_{h/2} D^2 \lambda_{h/2},
  \end{equation*}
  where the sum is over all $h$-elements in the coarse mesh, $K_h$,
  and $a_h^*$ is the harmonic mean of $a_{h/2}$ over each $K_h$, i.e.,
  \begin{equation*}
    a_h^*=\frac{2}{\frac{1}{a_{h/2}^+} + \frac{1}{a_{h/2}^-}}.
  \end{equation*}
  The second-order difference in each coarse grid interval is given by
  \begin{equation*}
    D^2u_{h/2} := \frac{D u_{h/2}^+ - D u_{h/2}^-}{h/2} =
    \frac{(u_{h/2}^+)' - (u_{h/2}^-)'}{h/2},
  \end{equation*}
  and similarly for $\lambda_{h/2}$.
\end{theorem}

\begin{proof}
  By Lemma \ref{lem:discretesolutions}, we have
  \begin{equation}
    \label{eq:seconddiff}
    D^2u_{h/2} = D(1/a_{h/2}) = \frac{2}{h}\left(\frac{1}{a_{h/2}^+}-\frac{1}{a_{h/2}^-
      }\right) = \frac{2}{h}\frac{a_{h/2}^--a_{h/2}^+}{a_{h/2}^+a_{h/2}^-}.
  \end{equation}
  We want to compare this with the difference between first-order
  derivatives on grids with lengths $h/2$ and $h$,
  \begin{equation}\label{eq:gradtwogrids}
    (u_{h/2}^\pm)^\prime - u_h^\prime
    =\frac{1}{a_{h/2}^\pm}-\frac{1}{a_h} = \frac{a_h-a_{h/2}^\pm}{a_h
      a_{h/2}^\pm} = \pm  \frac{a_{h/2}^- - a_{h/2}^+}{2 a_h a_{h/2}^\pm},
  \end{equation}
  where in the last equality we used
  \begin{equation*}
    a_h=\frac{a_{h/2}^+ + a_{h/2}^-}{2}.
  \end{equation*}
  Combining \eqref{eq:seconddiff} and \eqref{eq:gradtwogrids}, we
  obtain
  \begin{equation}\label{eq:udifferentgrids}
    (u_{h/2}^\pm)' - u_h' = \pm \frac{h}{2} \frac{a_{h/2}^+ a_{h/2}^-}{2 a_h a_{h/2}^\pm} 
    D^2 u_{h/2}
    =\pm h\frac{a_{h/2}^\mp}{4a_h}D^2 u_{h/2}.
  \end{equation}

  Proceeding in a similar manner, we now prove that $\lambda_h$ and
  $\lambda_{h/2}$ satisfy the same relation, i.e.,
  \begin{equation}\label{eq:lambdadifferentgrids}
    (\lambda_{h/2}^\pm)' - \lambda_h' 
    =\pm h\frac{a_{h/2}^\mp}{4a_h}D^2 \lambda_{h/2}.
  \end{equation}
  We begin by computing $D^2 \lambda_{h/2}$ using Lemma
  \ref{lem:discretesolutions}. This yields
  \begin{align*}
    D^2\lambda_{h/2} = D\left(\frac{G_{h/2}}{a_{h/2}}\right) 
    &= \frac{2}{h}\left(\frac{G_{h/2}^+}{a_{h/2}^+} 
      -  \frac{G_{h/2}^-}{a_{h/2}^-}\right) \\
    &= \frac{2}{h}\left(\frac{G_h}{a_{h/2}^+} - \frac{G_h}{a_{h/2}^-}
      + \frac{G_{h/2}^+-G_h}{a_{h/2}^+} 
      + \frac{G_{h}-G_{h/2}^-}{a_{h/2}^-}\right)
  \end{align*}
  and, since $G_{h/2}^+-G_h=G_h-G_{h/2}^- = (G_{h/2}^+-G_{h/2}^-)/2,$
  we obtain
  \begin{equation}
    \label{eq:lambdaseconddiff}
    D^2\lambda_{h/2} = \frac{2}{h}\left(G_h\frac{a_{h/2}^- -
        a_{h/2}^+}{a_{h/2}^+a_{h/2}^-} + \frac{G_{h/2}^+ -
        G_{h/2}^-}{a_h^*} \right)
  \end{equation}
  where we recall, for the convenience of the reader, that
  \begin{equation*}
    \frac{1}{a_h^*} = \frac{1}{2} \left(\frac{1}{a_{h/2}^+} + \frac{1}{a_{h/2}^-}\right).
  \end{equation*}
  The difference between first-order derivatives on grids with lengths
  $h/2$ and $h$ is then given by
  \begin{equation}
    \label{eq:lambdafirstdiff}
    \begin{split}
      (\lambda_{h/2}^\pm)' - \lambda_h' =
      \frac{G_{h/2}^\pm}{a_{h/2}^\pm} - \frac{G_h}{a_h} =
      \frac{G_{h/2}^\pm-G_h}{a_{h/2}^\pm} +
      G_h\Big(\frac{1}{a_{h/2}^\pm} -\frac{1}{a_h}\Big) \\ =
      \frac{G_{h/2}^\pm-G_h}{a_{h/2}^\pm} \pm G_h\frac{a_{h/2}^- -
        a_{h/2}^+}{2a_ha_{h/2}^\pm}.
    \end{split}
  \end{equation}
  Combining equations \eqref{eq:lambdaseconddiff} and
  \eqref{eq:lambdafirstdiff} yields \eqref{eq:lambdadifferentgrids}:
  \begin{equation*}
    \begin{split}
      (&\lambda_{h/2}^\pm)' - \lambda_h' \mp h
      \frac{a_{h/2}^\mp}{4a_h}
      D^2\lambda_{h/2} \\
      &= \frac{G_{h/2}^\pm-G_h}{a_{h/2}^\pm} \pm G_h\frac{a_{h/2}^- -
        a_{h/2}^+}{2a_ha_{h/2}^\pm} \mp h
      \frac{a_{h/2}^\mp}{4a_h}\frac{2}{h}\left(G_h\frac{a_{h/2}^- -
          a_{h/2}^+}{a_{h/2}^+a_{h/2}^-} + \frac{G_{h/2}^+ -
          G_{h/2}^-}{a_h^*}
      \right)\\
      &= \frac{G_{h/2}^\pm-G_h}{a_{h/2}^\pm} \mp
      \frac{a_{h/2}^\mp}{2a_ha_h^*} (G_{h/2}^+ - G_{h/2}^-) \\
      &= \pm \frac{G_{h/2}^+ -
        G_{h/2}^-}{2}\left(\frac{1}{a_{h/2}^\pm}-\frac{a_{h/2}^\mp}{a_h
          a_h^*}\right)=0,
    \end{split}
  \end{equation*}
  where in the last equality we used
  \begin{equation*}
    a_h a_h^* = \frac{a_{h/2}^+ + a_{h/2}^-}{\frac{1}{a_{h/2}^+} +
      \frac{1}{a_{h/2}^-}} = a_{h/2}^+ a_{h/2}^-.
  \end{equation*}

  Finally, combining \eqref{eq:udifferentgrids} and
  \eqref{eq:lambdadifferentgrids}, the expressions for the differences
  of first-order derivatives on the nested grids for the primal and
  dual variables, we obtain
  \begin{equation*}
    a_{h/2}^\pm (D u_{h/2}^\pm - D u_h)
    (D \lambda_{h/2}^\pm - D \lambda_h) 
    = \left(\frac{h}{4}\right)^2 
    \frac{a_{h/2}^\pm(a_{h/2}^\mp)^2}{a_h^2} 
    D^2 u_{h/2} D^2 \lambda_{h/2},
  \end{equation*} 
  an expression for the integrand in the statement of the theorem. By
  summing over $K_h$, the coarse grid with step length $h$, and
  recalling that each coarse interval has two subintervals of length
  $h/2$, we obtain
  \begin{align*}
    \tilde{F}(h) 
    &= \sum_{K_h} \frac{h}{2} \left(\frac{h}{4}\right)^2
      \frac{a_{h/2}^-(a_{h/2}^+)^2+a_{h/2}^+(a_{h/2}^-)^2}{\big(
      \frac{a_{h/2}^+ + a_{h/2}^-}{2}\big)^2} D^2 u_{h/2} D^2 \lambda_{h/2} \\                 
    &= \sum_{K_h}
      \frac{h^3}{16}\frac{2}{\frac{1}{a_{h/2}^+}+\frac{1}{a_{h/2}^-}}
      D^2 u_{h/2} D^2 \lambda_{h/2} =:\sum_{K_h} \frac{h^3}{16}a_h^* D^2
      u_{h/2} D^2 \lambda_{h/2},
  \end{align*}
  where $a_h^*$ is the harmonic mean of $a_{h/2}$ over each coarse
  interval.
\end{proof}

Thus, given assumption \eqref{eq:freq-split-by-scales}, Theorem
\ref{thm:errorrepresentation} suggests a computable estimate for the
observable of the Galerkin error for problem
\eqref{eq:model-problem}. In practice, we estimate the observable for
the Galerkin error by
\begin{equation}
  \label{eq:computable-error-simple-eg}
  \mathcal{E}^h(g) \leq C' F(h) \leq  C E_{est}^h(g),
\end{equation}
where the estimator,
\begin{equation*}
  E^h_{est} (g) := 
  \sum_{K_h} \frac{h^3}{16}a_h^* \left| D^2 u_{h/2} D^2 \lambda_{h/2} \right|,
\end{equation*}
is a sum over $h$-elements, $K_h$, $C^\prime$ is a constant determined
from \eqref{eq:estimator-bounded}, and $C=C(\omega)$ is a positive
random variable. Note that the functional $E^h_{est}$ depends
implicitly on $g$ through the dual variable $\lambda$, and here, and
in the sequel, it will be useful to include the superscript $h$ to
distinguish between the pathwise Galerkin errors of different mesh
sizes. With these ideas in mind, we investigate in \S
\ref{sec:determining-c} the $C$ appearing in
\eqref{eq:computable-error-simple-eg} and determine that a constant
factor is suitable for producing a reliable estimator.

\begin{remark} \rm
  \label{rmk:using-2h-instead-of-h}
  The estimator $E_{est}^h(g)$, based on $F(h)$, relies on the
  computed solutions $u_{h/2}$ and $\lambda_{h/2}$. From an
  applications perspective, it is not unreasonable to use
  $E_{est}^{2h}$ instead of $E_{est}^h$ as an upper-bound predictor
  for $\mathcal{E}^{h}$. We note that, for the simple model problem,
  the ratio $\tilde{F}(2h)/\tilde{F}(h) \approx 2$ since both
  $D^2 u_h$ and $D^2 \lambda_h$ are $O_{\prob}(h^{-\frac{1}{2}})$.
\end{remark}

\subsection{Numerical tests for the simple model problem}
\label{sec:determining-c}

In the previous section, we derived a computable estimator for the
Galerkin error observable given by
\eqref{eq:computable-error-simple-eg} where $C = C(\omega)$ is a
random positive constant. Next, we determine a choice of $C$ that
produces a reliable estimator by sampling the ratio of
$\mathcal{E}^h(g)$ to $E^{h}_{est}(g)$ for different observables for
the Galerkin error for the simple model problem. In numerical
experiments, we observe that the sample mean of these realizations
provides a reasonable factor in
\eqref{eq:computable-error-simple-eg}. We consider the observables
$\mathcal{E}^{h}(1)$ and $\mathcal{E}^{h}(\delta_{\frac{1}{2}})$ that
correspond to computing the integral
$\mathcal{E}^{h}(1) = \int_0^1 (u - u_{h}) \dd x$ and the difference of
point values
$ \mathcal{E}^{h}(\delta_{\frac{1}{2}}) = u(0.5) - u_{h}(0.5), $
Respectively.

We compute the desired samples of $\mathcal{E}^{h}$, using a reference
solution on a fine mesh, $h = 2^{-17}$, and the associated estimator,
$E^{h}_{est}$, for $M=10^4$ sample paths. Here we take $M$ large to
ensure that the statistical error is small and emphasize that, in
practice, only $O(1)$ samples are needed to have a good hold on the
statistical error in our error estimates. Each realization of the
conductivity, $a$, is generated on an extremely fine mesh and then a
quadrature mesh size $k = 2^{-22}$ is used to ensure that the
quadrature error is negligible. Figure \ref{fig:distributionC} shows
the sample distribution of $C$, where the areas of the histogram bins
have been normalized to $1$, corresponding to the observable $g = 1$
(top row) and the observable $g = \delta_{\frac{1}{2}}$ (bottom row)
computing using the values $h = 2^{-10}$ (left column) and
$h = 2^{-12}$ (right column).  The sample mean, $\mu$, and the sample
standard deviation, $\sigma$, are also displayed. We observe, for both
observables, the sample mean is approximately $2$.

\begin{figure}[]
  \centering \subfloat{%
    \includegraphics[width=0.99\textwidth]{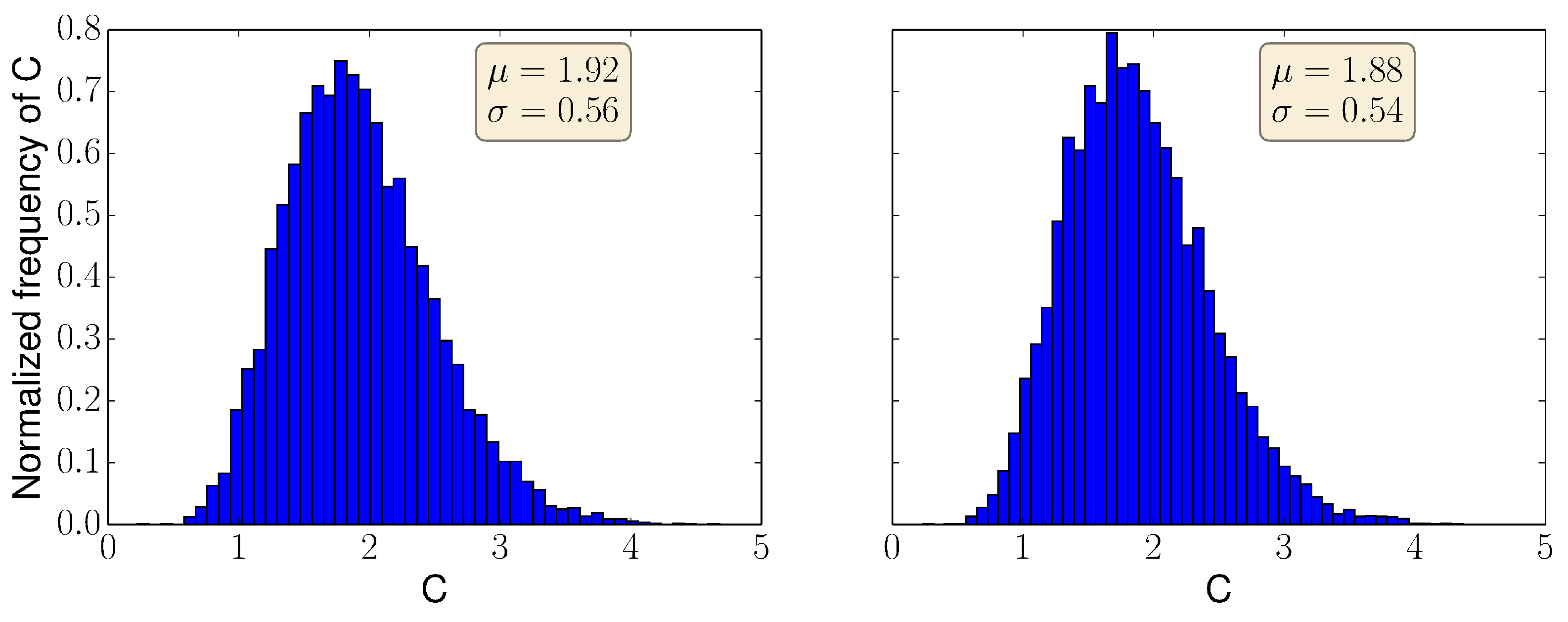}
  }\\%
  \hfill \subfloat{%
    \includegraphics[width=0.99\textwidth]{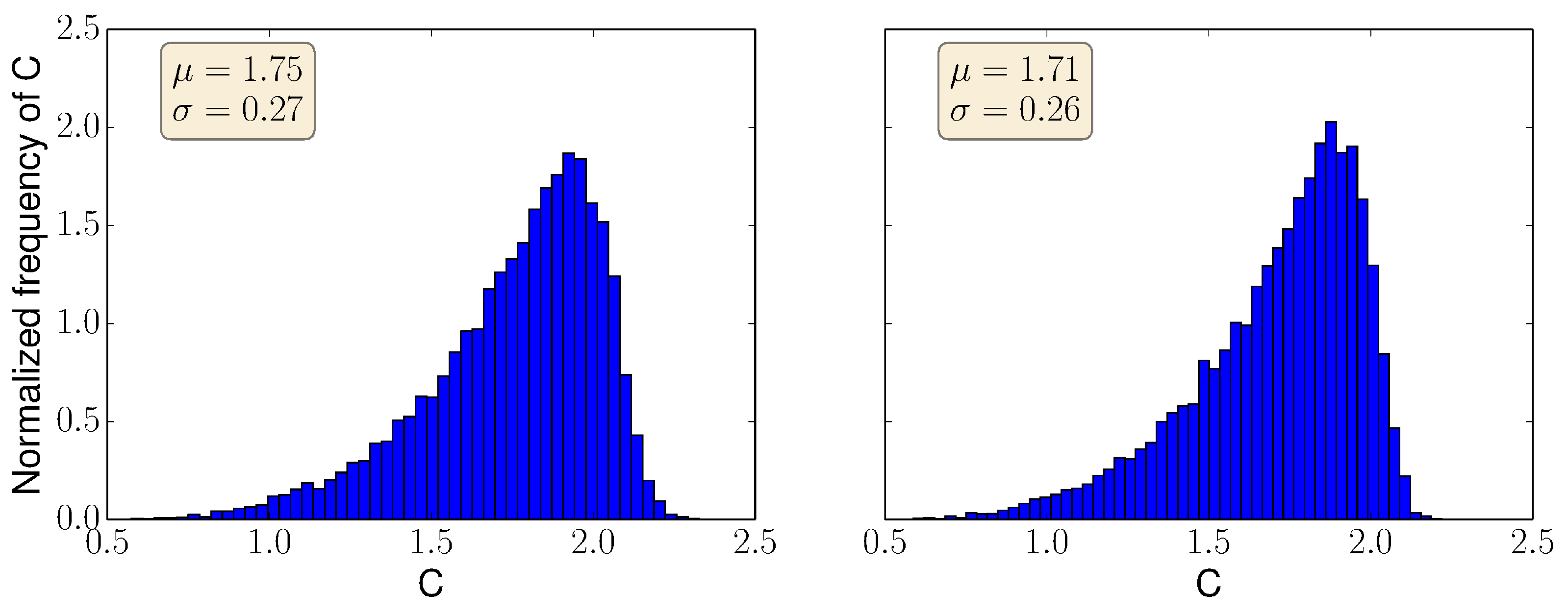}
  }%
  \caption{Histogram sample distribution of
    $C = |\mathcal{E}^h(g)|/E^{h}_{est}(g)$ for different combinations
    of mesh sizes and observables. All histograms are computed from
    $10^4$ realizations of $C$ and the areas of the histogram bins are
    normalized to 1.  The sample distributions for $\mathcal{E}^h(1)$
    (top row) and $\mathcal{E}^h(\delta_{\frac{1}{2}})$ (bottom row)
    are based on finite element solutions of $u_h$ on mesh sizes
    $h=2^{-10}$ (left-hand column) and $h = 2^{-12}$ (right-hand
    column).  The results suggests a reliable estimator for the simple
    model problem \eqref{eq:model-problem} can be attained using a
    constant $C=2$.}
  \label{fig:distributionC}
\end{figure}

In Figure~\ref{fig:comparisonGisgeneric}, we test the fit of using
$C = 2$, independently of $\omega$, with the estimators for two
different observables, $g = 1$ (top row) and
$g = \delta_{\frac{1}{2}}$ (bottom row). This comparison is done in
mean, based on 250 samples, in the left-hand column and is done
pathwise in the right-hand column, corresponding to an error and
estimator pair drawn uniformly at random from the previously generated
realizations. Choosing $C = 2$ in
\eqref{eq:computable-error-simple-eg} reliably estimates the
observable Galerkin error for the generic observables $g=1$ and
$g = \delta_{\frac{1}{2}}$. Once again, to explore the limits of the
estimator, we consider the observable $g = cos(2\pi\cdot)$, a
non-generic choice selected to exhibit cancellations between the local
indicators. In Figure \ref{fig:comparisonGiscos}, we plot the error
and estimator, using $C = 2$, in expectation, based on 250 samples,
(left-hand panel) and pathwise (right-hand panel). As anticipated, the
estimator, which is the sum of local error indicators, can become
small arbitrarily due to cancellations. In the next section, \S
\ref{sec:2d-numer-exper}, we introduce and test an estimator for a
two-dimensional model problem, but first we provide heuristics to
motivate our use of the pathwise Galerkin error.

\begin{figure}[]
  \centering \subfloat{%
    \includegraphics[width=0.49\textwidth]{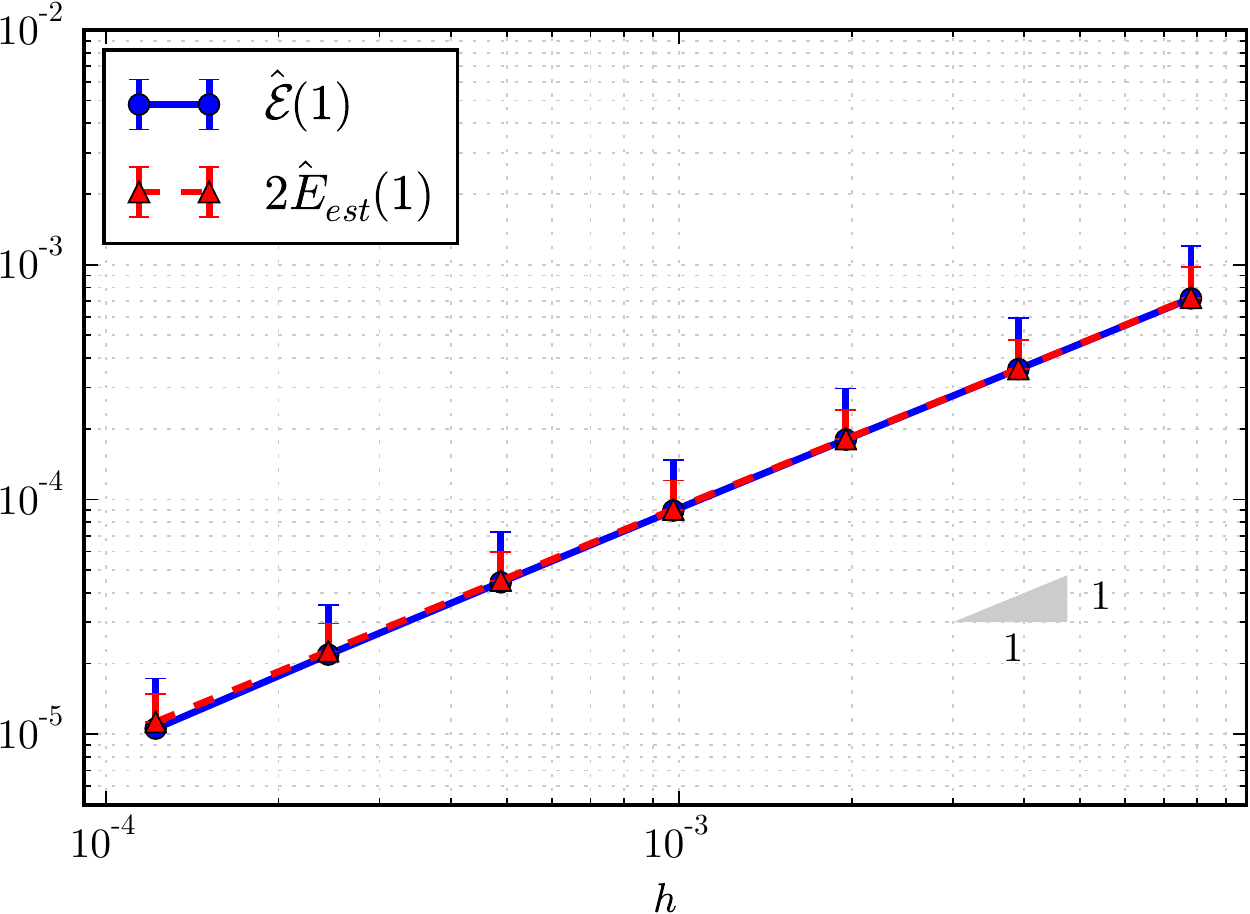}
  }%
  \hfill \subfloat{%
    \includegraphics[width=0.49\textwidth]{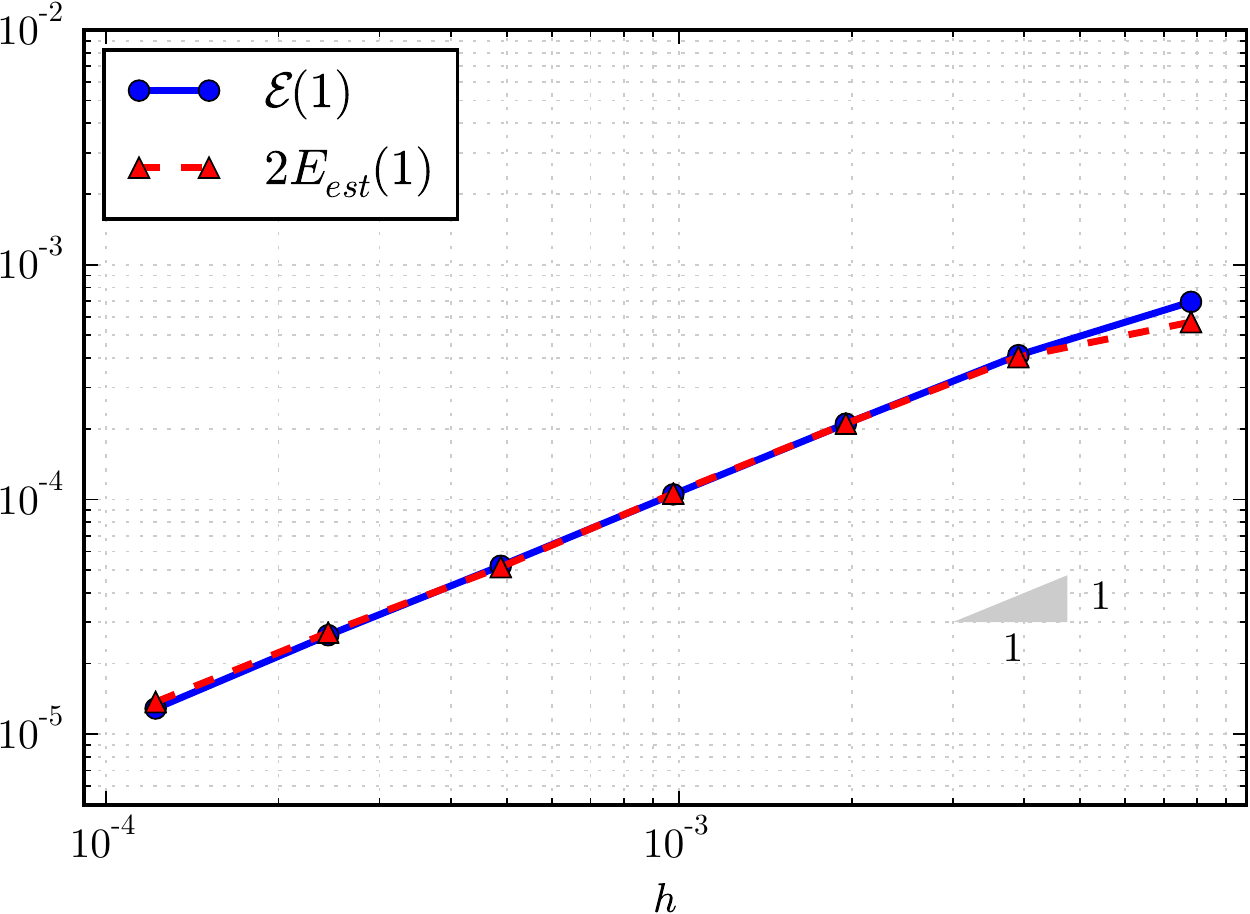}
  }%

  \subfloat{%
    \includegraphics[width=0.49\textwidth]{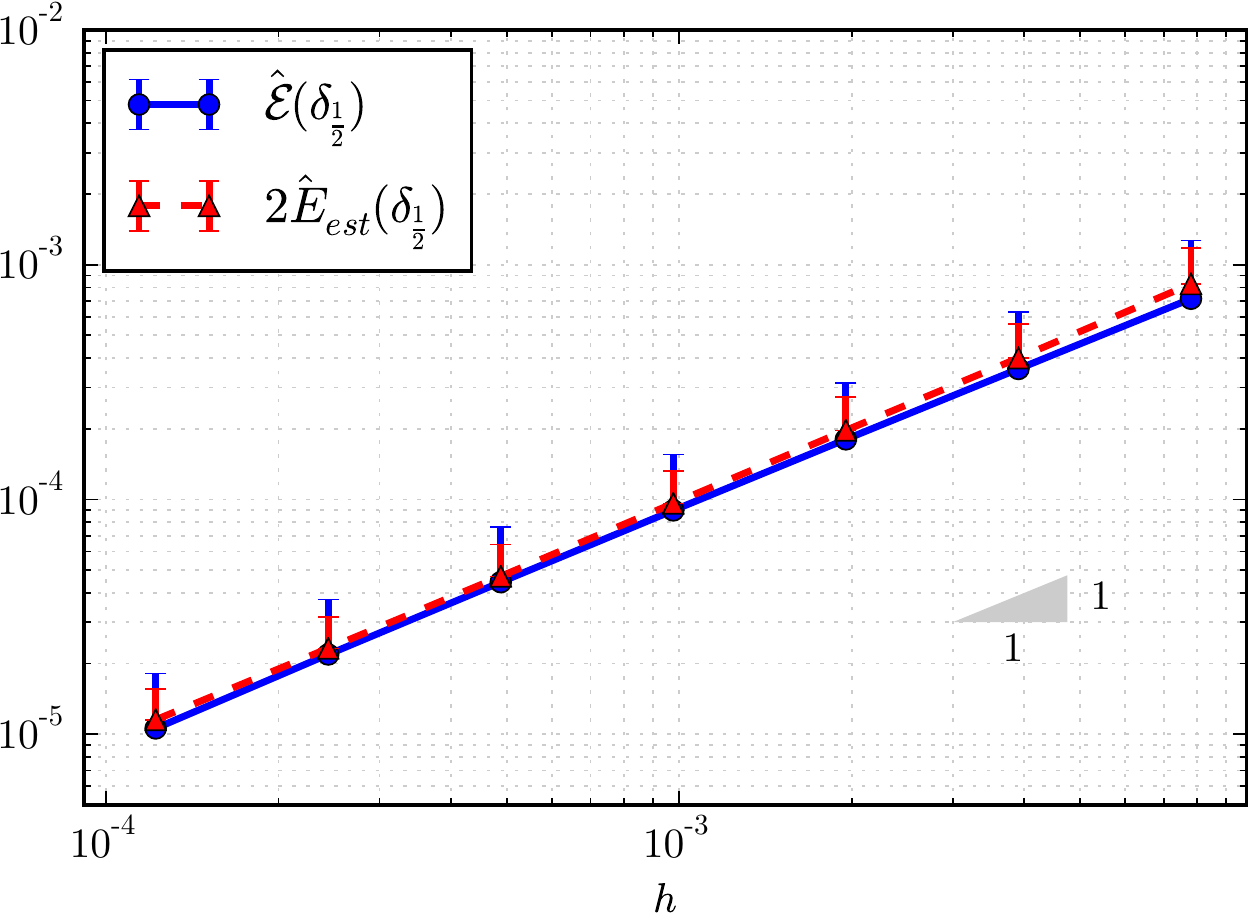}
  }%
  \hfill \subfloat{%
    \includegraphics[width=0.49\textwidth]{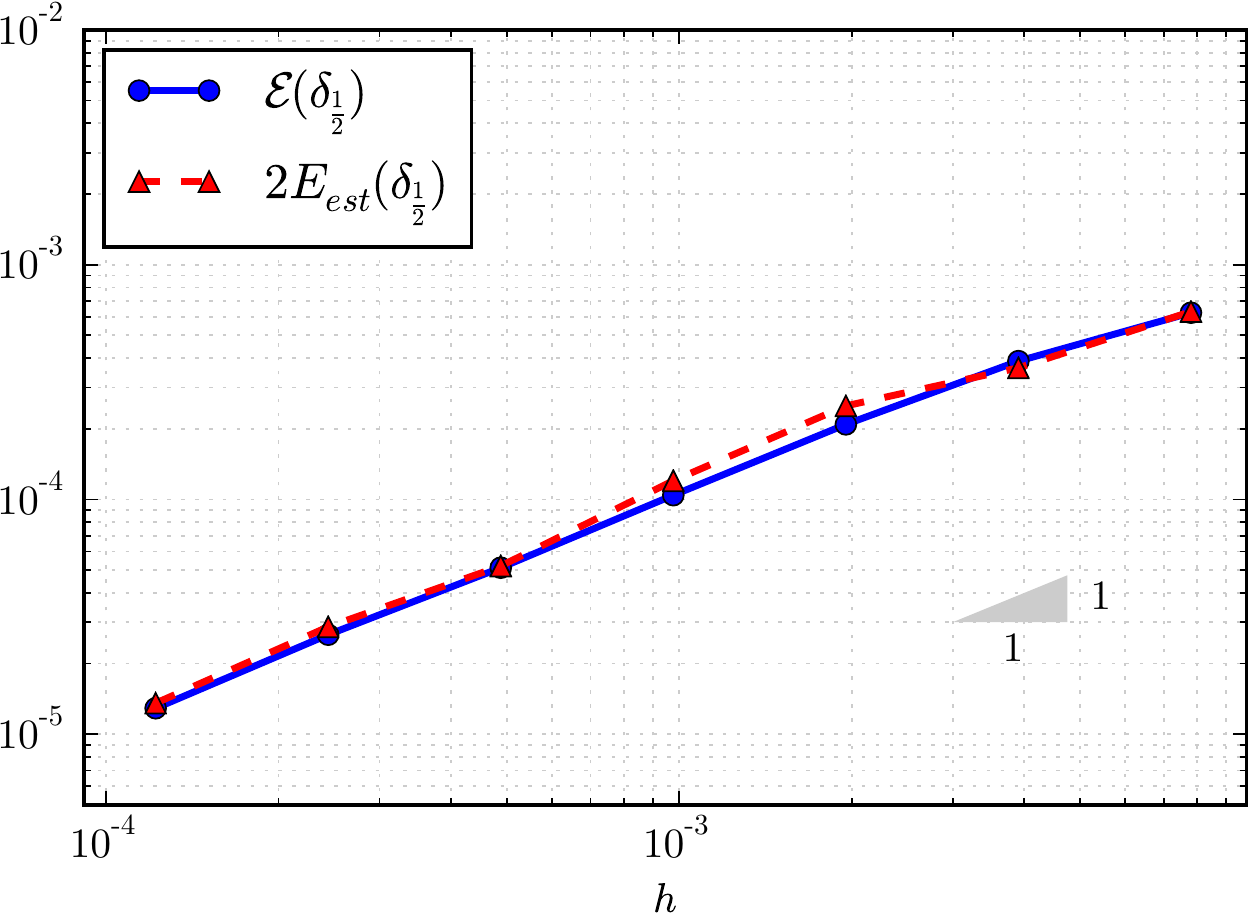}
  }%
  \caption{Using a constant $C = 2$, independent of $\omega$, produces
    a reliable estimator for the Galerkin error for the simple model
    problem \eqref{eq:model-problem} for the generic observables $g =
    1$ (top row) and $g = \delta_{\frac{1}{2}}$ (bottom row) both in
    expectation, based on 250 samples with error bars indicating two
    standard deviations, (left-hand column) and pathwise for one
    realization (right-hand column).}
  \label{fig:comparisonGisgeneric}
\end{figure}

\begin{figure}[]
  \centering \subfloat{%
    \includegraphics[width=0.49\textwidth]{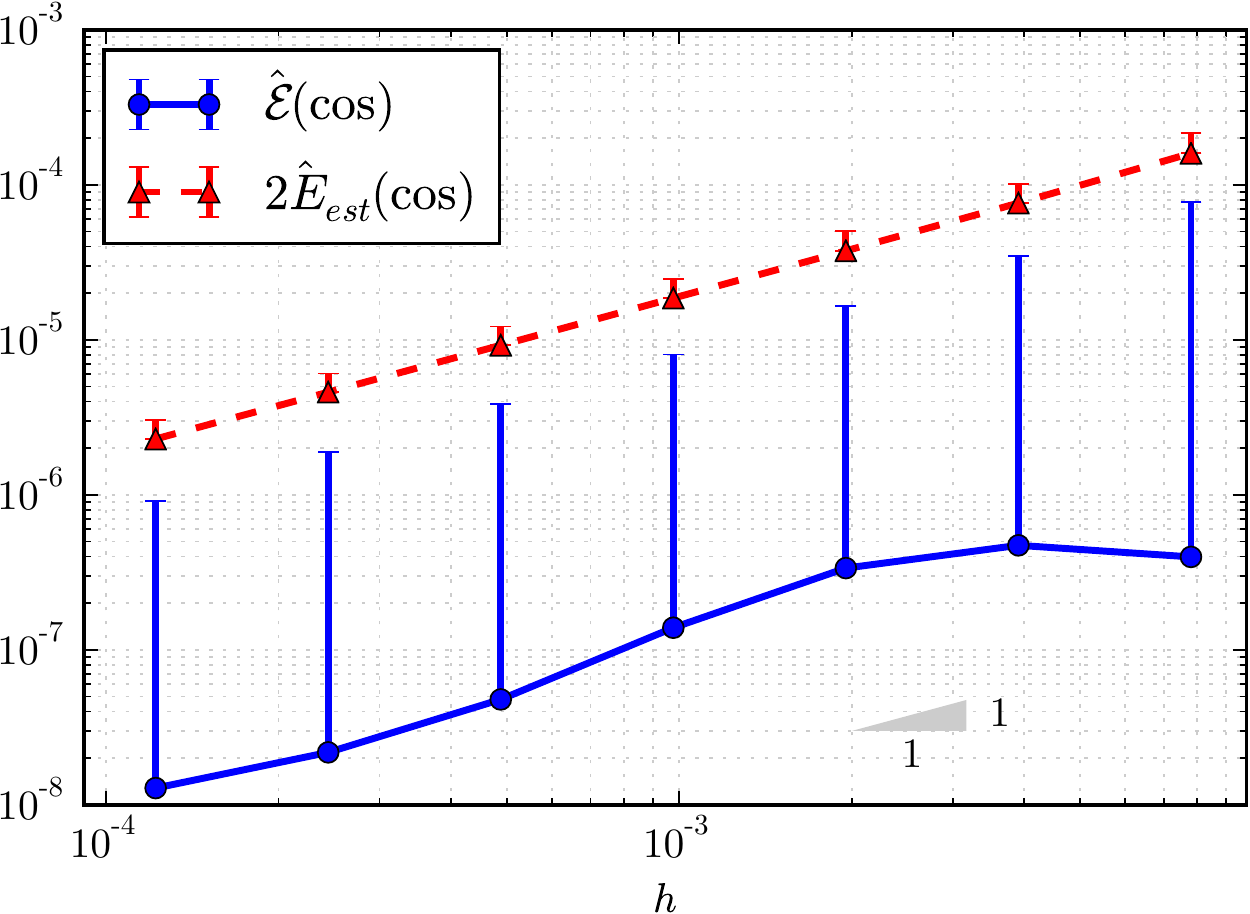}
  }%
  \hfill \subfloat{%
    \includegraphics[width=0.49\textwidth]{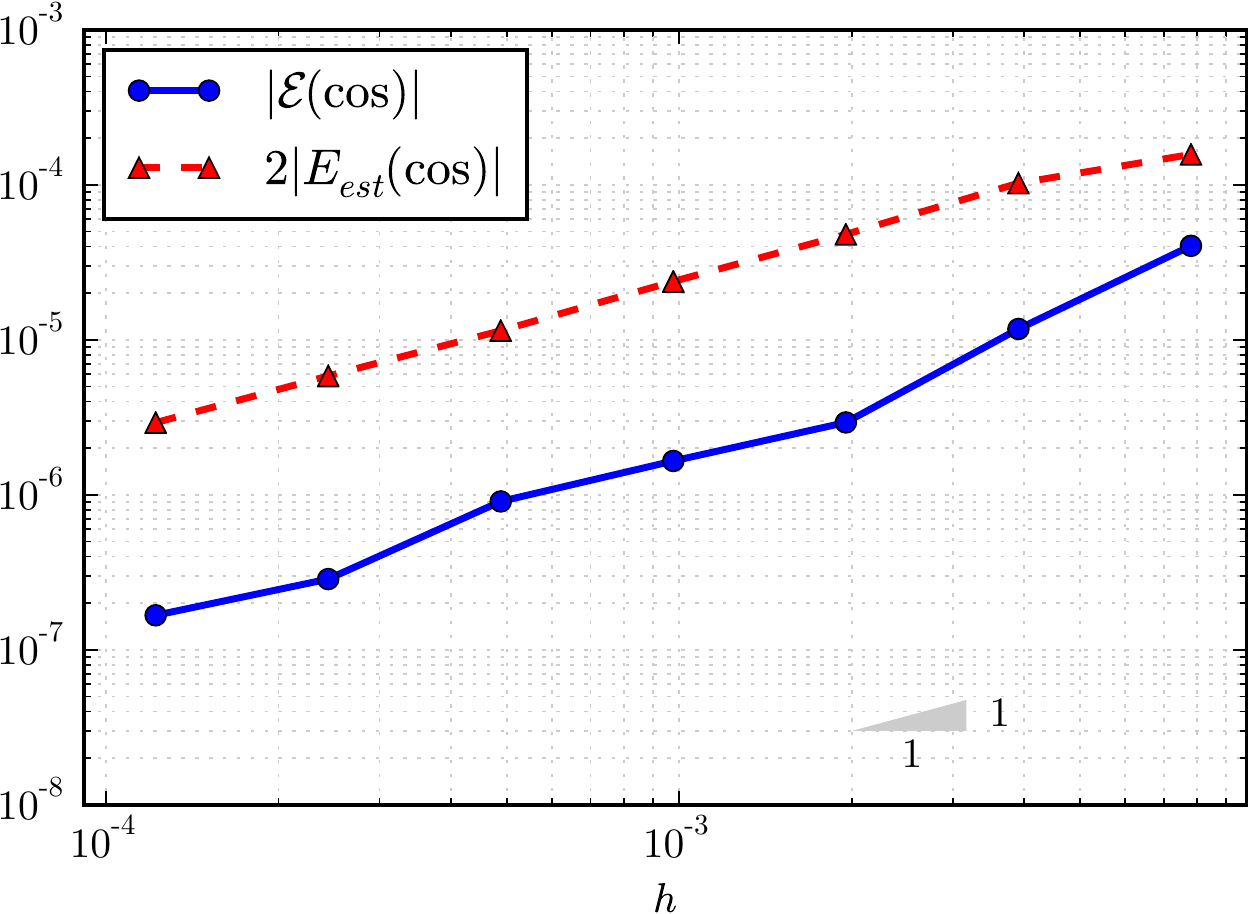}
  }%
  \caption{The expected estimator (left-hand panel) and the pathwise
    estimator for one realization (right-hand panel) for the Galerkin
    error for the simple model problem for the pathological
    observable, $g = \cos(2\pi\cdot)$, can become small due to
    cancellations. Here the expected estimator is based on 250 samples
    with error bars indicating two standard deviations.}
  \label{fig:comparisonGiscos}
\end{figure}

\subsection{The Pathwise Galerkin error rate is arbitrarily close to
  the expected Galerkin error rate}
\label{appendix:pathwise-Galerkin-vs-mean}

In the preceding we studied the pathwise Galerkin error rather than
the expected Galerkin error. Due to possible stochastic cancellations,
one might suspect that the convergence rate of the expected Galerkin
error to be higher than that of the pathwise Galerkin error. We now
study the relationship between these rates in detail for the following
random PDE that is related to the model
problem~\eqref{eq:model-problem}, but slightly easier to analyze in
this context:
\begin{equation}\label{eq:rpdeWiener}
  -(a(\omega,x) u'(\omega,x))' = 0,
\end{equation}
$\as$ for $(\omega, x) \in \Omega \times [0,1]$ subject to the
boundary conditions $u(\omega, 0)=0$ and
$a(\omega,1)u'(\omega,1) = 1$, $\as$, with random field $a= e^{W}$,
where $W$ denotes a standard Wiener process. The next theorem shows
that under certain assumptions on the observable applied to
problem~\eqref{eq:rpdeWiener}, the pathwise Galerkin error rate is
arbitrarily close to the expected Galerkin error rate.

\begin{theorem}
  Consider the random PDE~\eqref{eq:rpdeWiener} and let $u_h$ denote
  the discrete solution on the Galerkin space defined in
  Section~\ref{sec:comp-error-simple}. For a given observable $g$,
  assume that the function $G$, as defined in \eqref{eq:Gdef}, is in
  $C^1([0,1])$, non-negative, and strictly positive on a positive
  measure subset of $[0,1]$. Then, for any $\delta >0$, there exists a
  pathwise constant $c \in [0, \infty)$ such that
  \begin{equation*}
   \abs{(u-u_h, g)} \leq c h^{1-\delta},
  \end{equation*}
  and another constant $c'>0$ such that
  \begin{equation}\label{eq:expGalerkin}
    \lim_{h \to 0} h^{-1} \E{(u-u_h, g)} 
    = \int_0^1 \frac{e^{x/2} G(x)}{6} \dd x >  c^\prime.
  \end{equation}
\end{theorem}

\begin{proof}
  One may verify that the derivatives of the discrete primal and dual
  solutions for the considered problem satisfy
  equation~\eqref{eq:discreteDerivatives} in
  Lemma~\ref{lem:discretesolutions}.  Introducing the mesh points
  $x_n = nh$, for $n =0,1,\ldots h^{-1} =:N$, the pathwise Galerkin
  error, $\mathcal{E}(g) = (u-u_h,g)$, takes the form
  \begin{equation}\label{eq:Holdersplit}
    \begin{split}
       \mathcal{E}(g) &=\int_0^1 a(u-u_h)'(\lambda-\lambda_h)'\dd x\\
      &=\int_0^1 a(a^{-1}-a_h^{-1})^2 G_h(x)\dd x
      +\int_0^1\pr{\frac{1}{a}-\frac{1}{a_h}}(G-G_h) \dd x\\
      &=\sum_{n=0}^{N-1} \int_{x_n}^{x_{n+1}}
      a(x)\left(a^{-1}(x)-a(\xi_n)^{-1}\right)^2 G_h(x)\dd  x\\
      &\qquad + \sum_{n=0}^{N-1}
      \int_{x_n}^{x_{n+1}}\pr{\frac{1}{a}-\frac{1}{a_h}}(G-G_h) \dd
      x\\
      & =: I+II,
    \end{split}
  \end{equation}
  where $G_h(x):=\int_{x_n}^{x_{n+1}} G(y)\dd y$ and $a_h(x)=
  h^{-1}\int_{x_n}^{x_{n+1}} a(y)\dd y=a(\xi_n)$ when $x_n\leq x <
  x_{n+1}$, and for some path dependent $\xi_n\in[x_n,x_{n+1}]$. By
  $a\in C^{\frac{1}{2}-\delta}([0,1])$ and $G\in C^1([0,1])$, the
  error terms in \eqref{eq:Holdersplit} satisfy $\lvert I \rvert\leq c
  h^{1-2\delta}$ and $\lvert II \rvert\leq c h^{3/2-\delta}$, so the
  Galerkin error has the upper bound
  \begin{equation*}
    \abs{\mathcal{E}(g)}\le c h^{1-2\delta},
  \end{equation*}
  for a pathwise dependent constant $c \in [0, \infty)$ and $h$
  sufficiently small.

  For verifying inequality~\eqref{eq:expGalerkin}, we split, as
  above, the expected Galerkin error as
  \begin{equation}\label{eq:Galerkinsplit}
    \E{\mathcal{E}(g)} 
    =\E{\int_0^1 \pr{\frac{1}{a}-\frac{1}{a_h}}(G-G_h)
    \dd x} + \E{\int_0^1 a(a^{-1}-a_h^{-1})^2 G_h \dd x}.
  \end{equation}
  Using the representation
  $a^{-1}(x_n+y)=e^{-W(x_n+y)}= e^{-W(x_n)-(W(x_n+y)-W(x_n))},$ for
  $y \in [0, h]$ and $n=0,1,\ldots, N-1$, we proceed computing the
  second part of the expected Galerkin error \eqref{eq:Galerkinsplit}:
  \begin{equation}\label{eq:c*}
    \begin{split}
      &\E{\int_0^1  a(a^{-1}-a_h^{-1})^2 G_h \dd x}\\
      &\qquad = \sum_{n=0}^{N-1} \E{\int_{0}^{h}a(x_n+y)
        \left(a^{-1}(x_n+y)-a_h^{-1}\right)^2 G_h(x_n+y) \dd y}\\
      &\qquad = \sum_{n=0}^{N-1} \mexp
      \Bigg[e^{-W(x_n)}G_h(x_n)\int_{0}^{h} e^{W(x_n+y)-W(x_n)} \\
      & \qquad\qquad \times \pr{e^{-(W(x_n+y)-W(x_n))}-
        \frac{1}{h^{-1} \int_0^h e^{W(x_n+z)-W(x_n)} \dd z} }^2 \dd y \Bigg]\\
      &\qquad = \sum_{n=0}^{N-1} \E{e^{-W(x_n)}G_h(x_n)}
      \E{\int_{0}^{h}a(y)\left(a^{-1}(y)-a_h^{-1}\right)^2 \dd y},
    \end{split}
  \end{equation}
  where we used in the third equality that $W(x_n)$ and
  $W(x_n+y)-W(x_n)$ are independent Wiener increments. Similarly we
  have that the first part of the expected Galerkin error
  \eqref{eq:Galerkinsplit} can be bounded as
  \begin{equation}\label{eq:Galerkinbound}
    \begin{split}
      &\left| \E{\int_0^1 \pr{\frac{1}{a}-\frac{1}{a_h}}(G-G_h)
      \dd x} \right| \leq \max_{0\leq x \leq 1} \lvert G-G_h
      \rvert \E{\int_0^1
      \pr{\frac{1}{a}-\frac{1}{a_h}} \dd x} \\
      &\qquad = \max_{0\leq x \leq 1} \lvert G-G_h \rvert \sum_{n=0}^{N-1}
      \E{e^{-W(x_n)}} \E{\int_{0}^{h} \pr{\frac{1}{a}-\frac{1}{a_h}}
      \dd y}.
    \end{split}
  \end{equation}
  Furthermore,
  \begin{equation}\label{eq:errorint}
    \int_0^{h} a\pr{\frac{1}{a}-\frac{1}{a_h}}^2 \dd x
    = \int_0^{h}\pr{\frac{1}{a}- \frac{2}{a_h} + \frac{a}{a_h^2}} \dd x\\
    =\int_0^{h} \frac{1}{a} -\frac{1}{a_h} \dd x.
  \end{equation}
  We now show that the expected value of the integral in
  \eqref{eq:errorint} is given by
  $\frac{h^2}{6}+O(h^{5/2})$. The expected value of the
  first term in the last integral in \eqref{eq:errorint} is given by
  \begin{equation*}
    \E{\int_0^h \frac{1}{a} \dd x} =\int_0^h \E{e^{-W(x)}} \dd x =
    \int_0^h e^{x/2} \dd x = 2(e^{h/2}-1).
  \end{equation*}
  The expected value of the second term in the last integral in
  \eqref{eq:errorint} is given by
  \begin{equation}\label{eq:ahinv}
    \begin{split}
      &\E{\int_0^h \frac{1}{a_h} \dd x} = h\E{\frac{1}{a_h}} = h
      \E{\pr{\frac{1}{h} \int_0^h e^{W(x)} \dd x}^{-1}} \\
      &\qquad =  h \E{\pr{\underbrace{\frac{1}{h} \int_0^h \pr{e^{W(x)}-1}
      \dd x}_{=:z} +1}^{-1}}
      =h \E{1-z+z^2-\frac{z^3}{1+z}}.
    \end{split}
  \end{equation}
  As before, $\E{z}=\frac{2}{h}(e^{h/2}-1)-1$.  The next term is
  given by
  \begin{equation*}
    \begin{split}
      \E{z^2} &= \frac{1}{h^2}\E{\pr{\int_0^h (e^{W(x)}-1) \dd x}^2}\\
      &=\frac{1}{h^2}\E{\int_0^h\int_0^h(e^{W(x)+W(y)}-e^{W(x)}-e^{W(y)}-1)
      \dd x \dd y}.
    \end{split}
  \end{equation*}
  Using that $\E{(W(x)+W(y))^2} = x+y+2\min(x,y)$ and Taylor
  expansion give
  \begin{equation*}
    \begin{split}
      \E{z^2}&= \frac{1}{h^2}\int_0^h\int_0^h
      \pr{\frac{x+y}{2}+\min(x,y)-\frac{x}{2}-\frac{y}{2}} \dd x \dd y
      + O(h^2)\\
      &= \frac{1}{h^2}\int_0^h\int_0^h \min(x,y) \dd x \dd y + O(h^2)
      = \frac{h}{3} + O(h^2).
    \end{split}
  \end{equation*}
  The last term in \eqref{eq:ahinv} can be bounded using
  Cauchy-Schwarz:
  \begin{equation*}
    \left| \E{\frac{z^3}{1+z}}\right| \leq
    \sqrt{\E{z^6}}\sqrt{\E{\frac{1}{(1+z)^2}}}=
    \sqrt{\E{z^6}}\sqrt{\E{\frac{1}{a_h^2}}}.
  \end{equation*}
  Both factors can be treated using Jensen's inequality:
  \begin{equation*}
    \E{z^6}=\E{\pr{\frac{1}{h}\int_0^h(e^{W(x)}-1) \dd x}^6} \leq
    \frac{1}{h} \int_0^h \E{(e^{W(x)}-1)^6} \dd x = O(h^3)
  \end{equation*}
  and
  \begin{equation*}
    \E{\frac{1}{a_h^2}} = \E{\pr{\frac{1}{h}\int_0^h e^{W(x)} \dd
    x}^{-2}} \leq \E{\frac{1}{h}\int_0^h e^{-2W(x)} \dd x} =
    \frac{e^h-1}{h}= O(1).
  \end{equation*}
  Hence
  \begin{align}\label{eq:d*}
      \E{\int_0^h \pr{\frac{1}{a}-\frac{1}{a_h}}\dd x} 
      &= 2(e^{h/2}-1)-h\left(1-\left(\frac{2}{h}(e^{h/2}-1) -1\right)
        +\frac{h}{3}\right) + O(h^{5/2}) \notag \\
      &= \frac{h^2}{6} + O(h^{5/2}).
  \end{align}

  By equations~\eqref{eq:c*}, \eqref{eq:Galerkinbound},
  \eqref{eq:d*}, the equality $\E{e^{W(x)}} = e^{x/2}$, and that
  $G\in C^1([0,1])$, we conclude that
  \begin{equation*}
    \begin{split}
      \lim_{h\to 0} h^{-1} \E{(u-u_h,g)} =\frac{1}{6} \int_0^1
      e^{x/2} G(x) \dd x > c',
    \end{split}
  \end{equation*}
  where the inequality follows from $G$ being non-negative and
  strictly positive on a positive measure subset of $[0,1]$.
\end{proof}

%
\section{Two-dimensional numerical experiments}
\label{sec:2d-numer-exper}
%

We consider as a model the problem of finding
$u \in \mathcal{H}^{1}(\Gamma)$ for $\Gamma = (0,1)^2$ such that
\begin{equation}
  \label{eq:model-prob-2d}
  - \divergence(a(x) \nabla u(x)) = 1 \qquad \text{for } x \in \Gamma
\end{equation}
subject to the homogeneous boundary condition $u(x) = 0$ for all
$x \in \partial \Gamma$. We assume the conductivity $a$ is given and
that $\log a$ has an exponential covariance given by the two-point
covariance function
\begin{equation}
  \label{eq:covariance}
  \cov (x, y) = \sigma^2 e^{-\| x - y \|/\ell}
\end{equation}
for a given variance, $\sigma^2$, and correlation length, $\ell$. For
our numerical experiments, we use the circulant embedding method (see
for example \cite{AsmussenGlynn:2007}) to generate samples of a
lognormal field $a$ having the desired covariance structure. In Figure
\ref{fig:example-field}, one realization of $a(x,y)$, with parameters
$\ell = 0.2$ and $\sigma^2 = 1$, on a uniform mesh of size
$k = 2^{-10}$ is plotted (left-hand panel) together with the
corresponding finite element approximation, $u_h(x,y)$, of problem
\eqref{eq:model-prob-2d} with mesh size $h = 2^{-5}$ (right-hand
panel). The approximation uses piecewise linear elements,
$K \in \mathcal{T}^{h}$, the standard three-point triangulation of
$\Gamma$ mesh with size $h$.

\begin{figure}[]
  \centering \subfloat{%
    \includegraphics[width=0.49\textwidth]{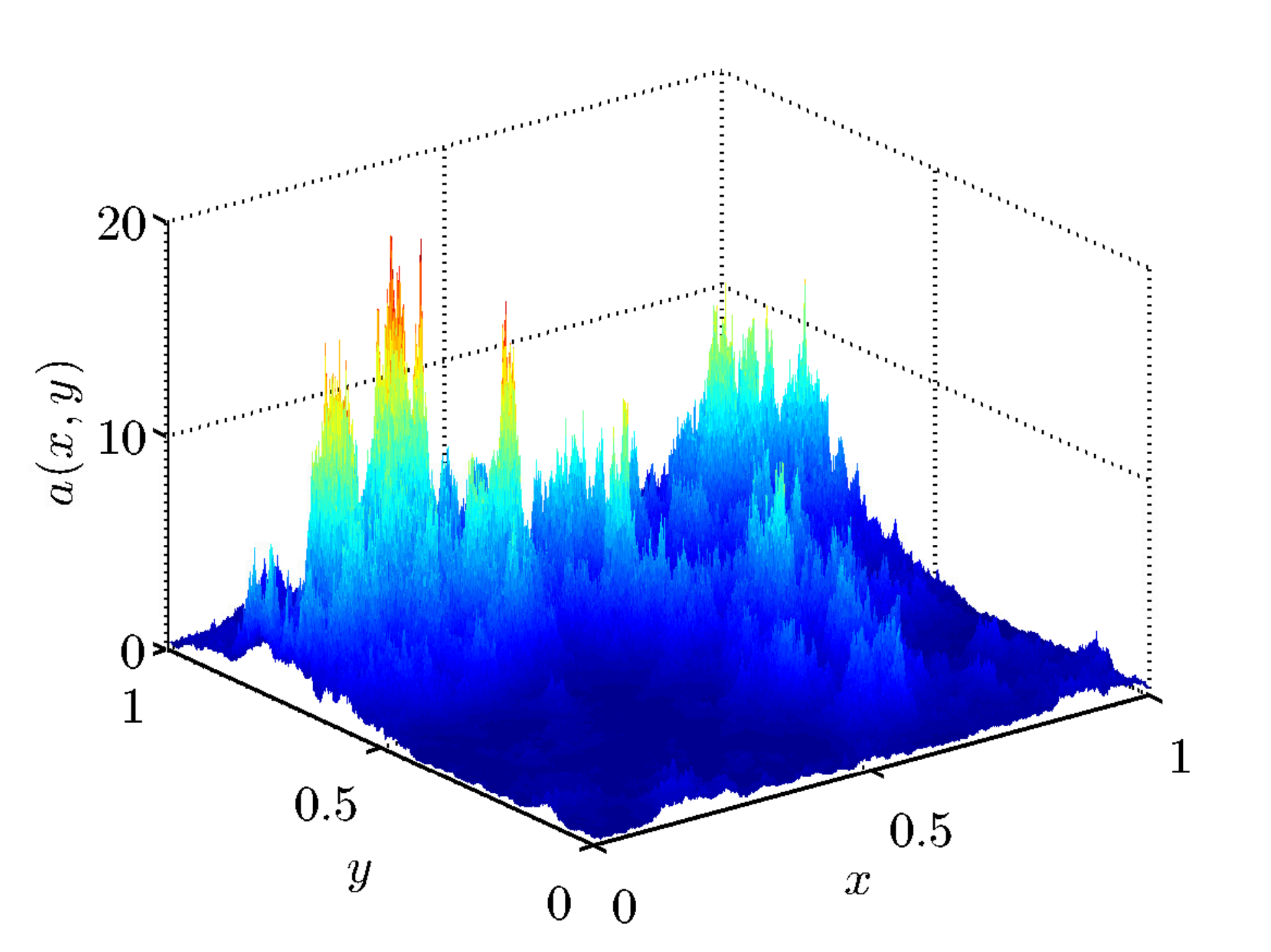}
  }%
  \hfill \subfloat{%
    \includegraphics[width=0.49\textwidth]{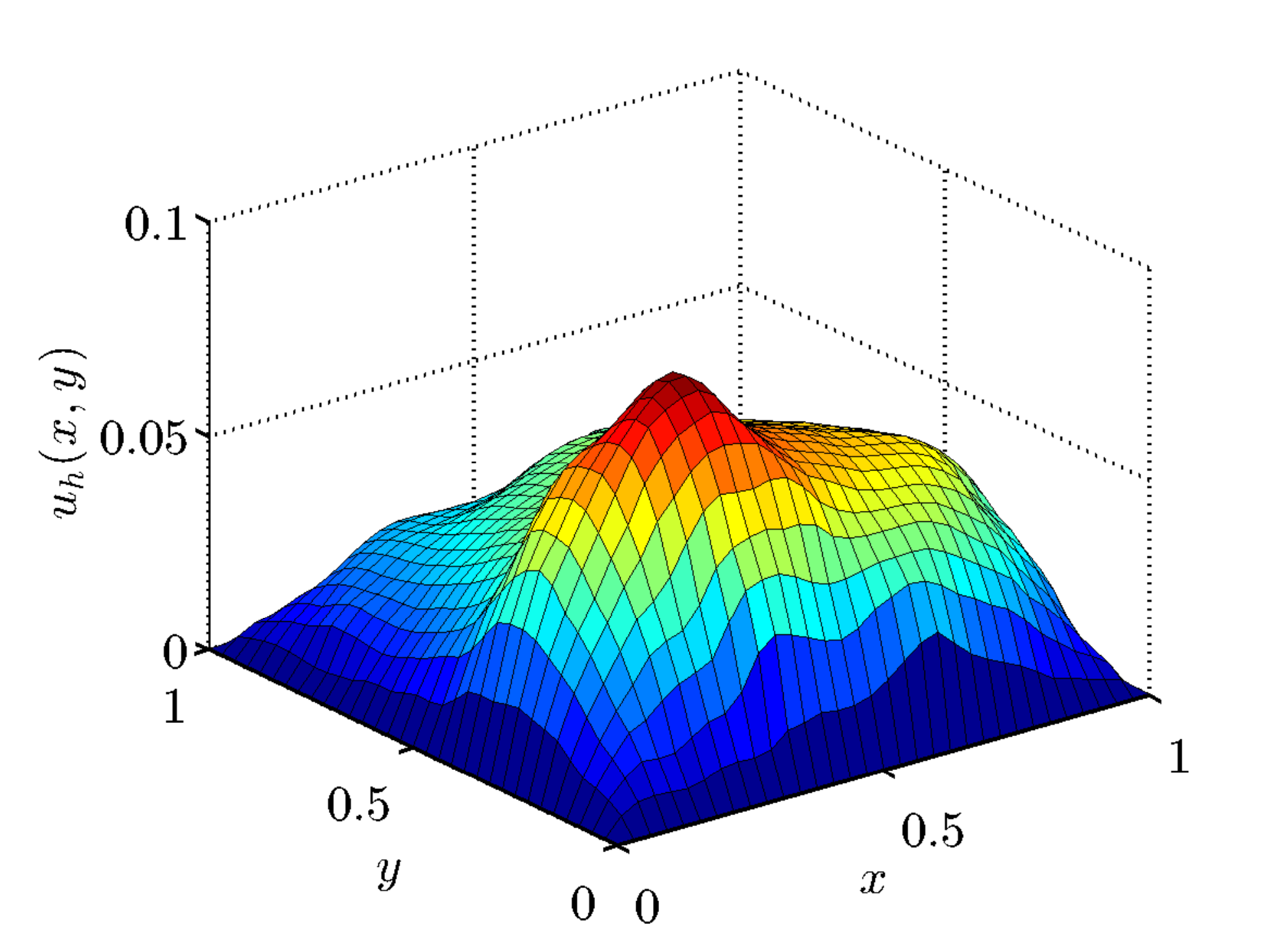}
  }%
  \caption{One sample of a rough lognormal conductivity, $a(x,y)$,
    with covariance \eqref{eq:covariance} with $\ell = 0.2$ and
    $\sigma^2 = 1$, generated using a circulant embedding method
    (left-hand panel) and the corresponding pathwise finite element
    approximation to problem \eqref{eq:model-prob-2d} (right-hand
    panel). For conductivities arising in groundwater flow problems
    the typically short correlation lengths involved motivate the use
    of MC finite element methods over stochastic Galerkin methods.}
  \label{fig:example-field}
\end{figure}
 
The considerations in \S \ref{sec:scales} suggest that, in practice,
\begin{equation*}
  \mathcal{E}^h(g) \leq C \int_\Gamma a_h 
  \left|\nabla ( u_{h/2} - u_h ) \cdot \nabla 
    (\lambda_{h/2} - \lambda_{h})\right| \dd x
\end{equation*}
admits an estimator based on local error indicators $a_h$,
$\nabla u_{h/2}$, $\nabla u_h$, $\nabla \lambda_{h/2}$, and
$\nabla \lambda_{h}$. We approximate this error by the estimator
\begin{equation*}
  E_{est}^h(g) :=  \frac{1}{2} h^2 \sum_{K} a_h 
  \sum_{i =1, 2} \left| \partial_{x_i} ( u_{h/2} - u_h) 
    \partial_{x_i} (\lambda_{h/2} - \lambda_h) \right|
\end{equation*}
where $K \in \mathcal{T}^{h}$ and $\partial_{x_i}$ is the derivative
in the $x_{i}$ direction. In addition, we compare with a second
estimator
\begin{equation*}
  E^{h}_{reg}(g) := \frac{1}{2} h^2 \sum_{K} \frac{h^2}{16} a_h 
  \sum_{i = 1,2} \left| D_i^2 u_{h} \cdot D_i^2 \lambda_{h} \right|.
\end{equation*}
The quantity $E^{h}_{reg}(g)$ is proven in \cite{MoonEtAl:2006} to be
an asymptotically exact estimate of the Galerkin approximation error,
as the mesh size tends to zero, in the case the solution is
sufficiently regular, i.e. if $a \in \mathcal C^1(\bar\Gamma)$,
$u\in\mathcal C^3(\bar\Gamma)$ and
$\lambda\in \mathcal C^3(\bar\Gamma)$, using bilinear finite elements
with possible hanging node refinements. The asymptotic exactness of
$E^{h}_{reg}$, which is related $E^h_{est}$ by constant factor of
order one, shows, for instance, that the error estimator does not miss
the contribution from the jumps of the derivatives at the edges,
although these are not explicitly present in the estimator.

In Figure \ref{fig:2d-err-est}, we plot $\mathcal{E}^h(1)$, using a
reference mesh of $h = 2^{-10}$, and the estimators $E^{h}_{est}(1)$
and $E^{h}_{reg}(1)$, all using a quadrature mesh, $k = 2^{-13}$. This
figure demonstrates that the estimator $E^{h}_{est}$, based on local
error indicators, gives an estimate of the error that is realized
pathwise provided that one has the quadrature error under control. We
choose $C = 1$ to show that the estimator holds with a nonrandom
constant, keeping in mind that this choice of constant is ad hoc and
does not represent a best fit. Further, we note that the Galerkin
error for the problem with rough lognormal conductivity is on the
order of $h$. On the other hand Figure~\ref{fig:2d-err-est} also
shows, as expected, that $E^h_{reg}$ misses the high frequency error
with a factor in the case that the solutions have low regularity,
analogously to the high- and low-frequency estimators in
Figure~\ref{fig:comp-low-freq-rough-smooth}.  Estimates for the
observable quadrature error are the subject of the next section.

\begin{figure}[]
  \centering
  \includegraphics[width=0.49\textwidth]{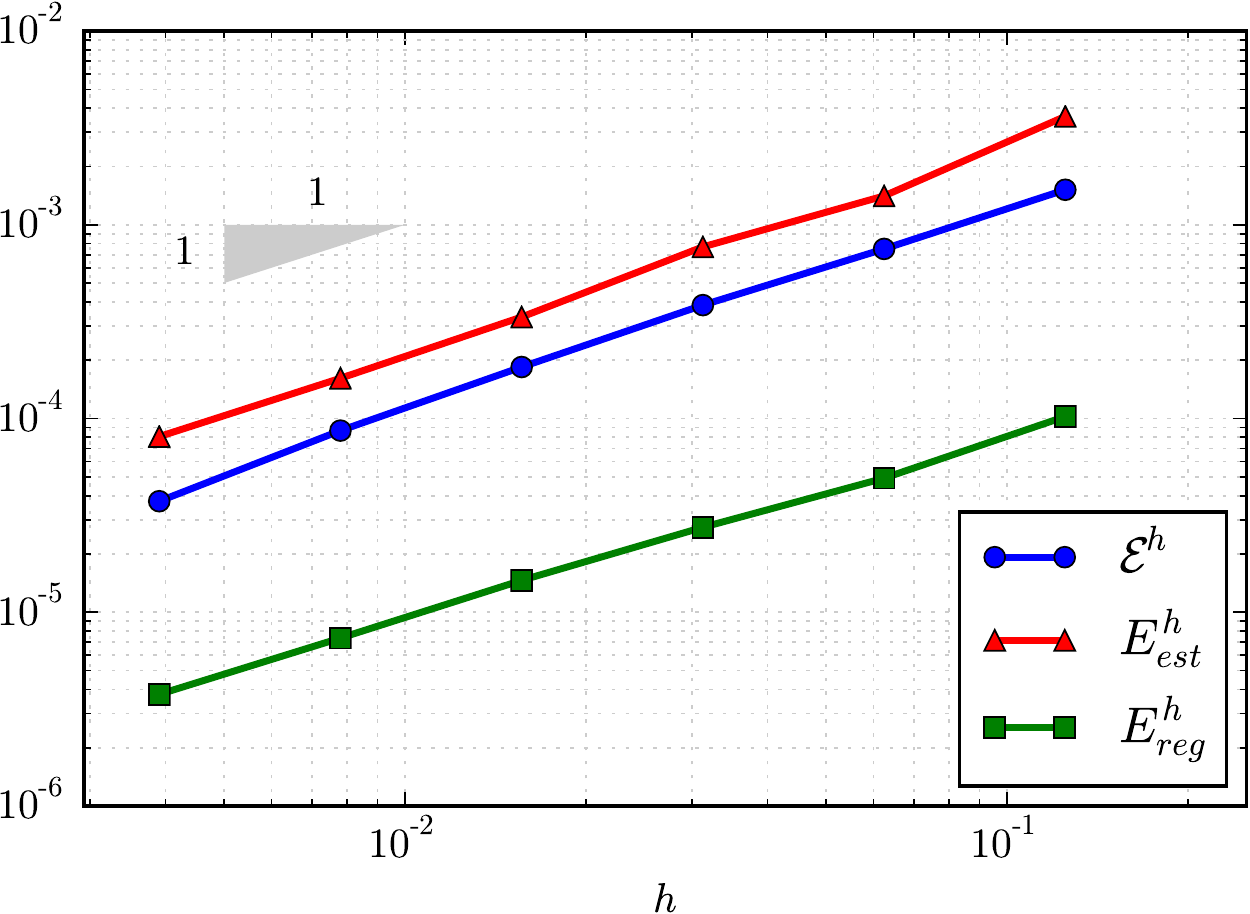}
  \caption{For the two-dimensional test problem
    \eqref{eq:model-prob-2d}, the estimator, $E_{est}^h(1)$,
    constructed from local error indicators reliably estimates the
    pathwise Galerkin error, $\mathcal{E}^h(1)$, with a constant
    factor independently of $\omega$. As expected, $E_{reg}^h(1)$
    misses the high frequency error with a factor, but the asymptotic
    exactness demonstrates that the estimator does not miss the
    contribution form the jumps of the derivatives at the edges of
    elements although they are not explicitly present.}
  \label{fig:2d-err-est}
\end{figure}

%
\section{Quadrature error}
\label{sec:quadrature-error}
%

The numerical experiments that we have presented thus far focused
solely on the Galerkin error at the expense of overkilling the
quadrature error. Next, we consider the problem of estimating the
expected quadrature error when computing an observable using a
piecewise linear finite element approximation of problem
\eqref{eq:rpde}. If the conductivity $a$ enjoys a certain degree of
regularity, then a quadrature rule can be chosen such that its error
contribution is negligible compared with the Galerkin error, as is
typically the case in the deterministic setting. In contrast, when a
conductivity $a$ has a rough lognormal distribution such as the
log-Brownian bridge distribution used in \eqref{eq:model-problem},
Figure \ref{fig:quad-rate} suggests that the quadrature error is on
the same order as the Galerkin error
(cf. Figure~\ref{fig:comp-low-freq-rough-smooth}, where the rate of
the Galerkin error is plotted). In the present setting, it is
therefore pertinent to construct an estimator for the expected
quadrature error committed in calculating an observable.

To discuss the quadrature error, we first present some additional
notation. Recall that a finite element solution, $u_h \in V_h$, to
\eqref{eq:rpde} satisfies
\begin{equation}
  \label{eq:forward-exact}
  \int_\Gamma a_h(x) \nabla u_h(x) \cdot \nabla v_h(x) \dd x 
  = \int_\Gamma f(x) v_h(x) \dd x \qquad \as
\end{equation}
for all $v_h \in V_h \subset V$, where $a_h$ is the spatial average of
$a$ over each $h$-element $K_h$, as defined in
\eqref{eq:spatial-avg-a}. Here, we assume that the finite element
solution, $u_h$, is computed without quadrature error. That is, we
assume that the components of the stiffness matrix, which are of the
form
\begin{equation*}
  \int_\Gamma a_h(x) \nabla \phi_i(x) \cdot \nabla \phi_j(x) \dd x,
\end{equation*}
are computed exactly. Likewise, an exact variational formulation of
the discrete dual, $\lambda_h \in V_h$, satisfies
\begin{equation}
  \label{eq:dual-exact}
  \int_\Gamma a_h(x) \nabla \lambda_h(x) \cdot \nabla v_h(x) \dd x 
  = \int_\Gamma g(x) v_h(x) \dd x \qquad \as 
\end{equation}
for all $v_h \in V_h \subset V$.

Together with \eqref{eq:forward-exact} and \eqref{eq:dual-exact}, we
consider formulations where a quadrature rule is used in the assembly
of the stiffness matrix in the finite element problem. We subdivide
the $h$-mesh, of elements $K_h$, into a quadrature $k$-mesh, of
elements $K_k \subset K_h$, with mesh size $k = k(h,n) = h2^{-n}$, for
some $n \in \NN := \{0,1,2,\dots\}$. On the $k$-mesh, we define the
functions $\bar{a}_{h,k}, \bar{f}_{h,k}:\Gamma \to \mathbf{R}$ which
are piecewise constant on all $K_k$ elements
\begin{equation*}
  \bar{a}_{h,k}(x) := \sum_{K_{k}} a(x_{K_k})  \ind{K_{k}} (x) \quad \text{and} \quad  \bar{f}_{h,k}(x) := \sum_{K_{k}} f(x_{K_k})  \ind{K_{k}} (x),
\end{equation*}
where $x_{K_k}$ is the midpoint (or a corner point) of the element
$K_k$. Taking averages over the $h$-elements, we have the piecewise
constant functions $a_{h,k}$ and $f_{h,k}:\Gamma \to \mathbf{R}$
defined by
\begin{equation*}
  a_{h,k}(x) := \sum_{K_k \subset K_h} a(x_{K_k}) \left|K_k\right| / \left|K_h\right| \quad \text{and} \quad 
  f_{h,k}(x) := \sum_{K_k \subset K_h} f(x_{K_k}) \left|K_k\right| / \left|K_h\right|,
\end{equation*}
for $x \in K_h$, where $\left| K \right| = \int_{K} \dd x$ is the size
of the element $K$.  Then in variational form, we seek
$u_{h,k} \in V_h$ such that
\begin{equation}
  \label{eq:forward-with-quad}
  \int_\Gamma a_{h,k}(x) \nabla u_{h,k}(x) \cdot \nabla v_h(x) \dd x 
  = \int_\Gamma  f_{h,k}(x) v_h(x) \dd x \qquad \as 
\end{equation}
for all $v_h \in V_h$ and, similarly, we seek $\lambda_{h,k} \in V_h$
such that
\begin{equation*}
  \int_\Gamma a_{h,k}(x) \nabla \lambda_{h,k}(x) \cdot \nabla v_h(x) \dd x 
  = \int_\Gamma g v_h \dd x \qquad \as
\end{equation*}
for all $v_h \in V_h \subset V$, which corresponds to using the
quadrature rule for calculating the components of the stiffness
matrix. As the regularity of $a$ is low, the midpoint rule with a
fixed number of points per $h$-element is a reasonable strategy for
approximating the spatial average of $a$.

We are interested in estimating
\begin{equation*}
  \mexp \mathcal{Q}(g) := \mexp \left[(g , u_h - u_{h,k})\right],
\end{equation*}
the expected quadrature error in the finite element approximation of
\eqref{eq:rpde}. For a finite sample of size $M$, we will consider the
MC approximation
\begin{equation*}
  \hat{\mathcal{Q}}^{h,k}(g) = \frac{1}{M} \sum_{m=1}^{M} 
  \left( g, u_{h}(\omega_m) - u_{h,k}(\omega_m)\right).
\end{equation*}
The statistical error committed in the MC approximation is
$O_P(\sigma_s / \sqrt{M})$ where $\sigma_s$ is the standard deviation
of $(g, u_h - u_{h,k})$.  This extra statistical error contribution
can be made small, relative to the contributions from the Galerkin
error and the quadrature error, by choosing sufficiently large $M$. In
the next theorem we demonstrate that once again an assumption on
scales and a telescoping argument can be used to obtain an estimator
for the desired expected quadrature error.

\begin{theorem}\label{thm:quadScales}
  Assume that for positive constants $C$ and $\gamma$,
  \begin{equation}
    \label{eq:assumption-on-scales-quadrature}
    \left| \mexp \left[ \int_\Gamma (a_{h,k} - a_{h, k/2}) \nabla 
        u_{h,k} \cdot \nabla \lambda_{h} \dd x \right] \right| < C k^\gamma,
  \end{equation}
  that for all sufficiently small $h,k>0$,
  \begin{equation*}
    \lim_{j \to \infty} \mexp\left[ 
      \int_\Gamma (a_{h,k2^{-j}} - a_h) \nabla u_{h,k} \cdot
      \nabla \lambda_h \dd x \right] = 0,
  \end{equation*}
  and that
  \begin{equation}\label{eq:rlimit}
    \lim_{k\rightarrow 0} k^{-\gamma}\mexp \left[ 
      \int_\Gamma (f-f_{h,k})\lambda_h \dd x \right]=0.
  \end{equation}
  Then there is a constant $\hat C$ such that
  \begin{equation*}
    \left\lvert\mexp \left[ \int_\Gamma g (u_h - u_{h,k}) \dd x \right] \right\rvert
    \leq \frac{\hat C k^\gamma}{1-2^{-\gamma}}.
  \end{equation*}
\end{theorem}

{\em Proof}. Using \eqref{eq:dual-exact}, \eqref{eq:forward-exact},
and \eqref{eq:forward-with-quad}, the expected quadrature error can be
expressed as
\begin{align*}
  \mexp \left[ \int_\Gamma g (u_h - u_{h,k}) \dd x \right]
  & = \mexp \left[ \int_\Gamma (a_h \nabla \lambda_h) \cdot \nabla (u_h - u_{h,k}) \dd x \right]\\
  & = \mexp \left[ \int_\Gamma a_h \nabla u_h \cdot \nabla \lambda_h
    \dd x - \int_\Gamma a_h \nabla
    u_{h,k} \cdot \nabla \lambda_h \dd x \right] \notag\\
  & = \mexp \left[ \int_\Gamma f \lambda_h \dd x - \int_\Gamma a_h
    \nabla u_{h,k} \cdot \nabla
    \lambda_h \dd x \right] \notag \\
  & = \mexp \left[ \int_\Gamma f_{h,k} \lambda_h \dd x - \int_\Gamma
    a_h \nabla u_{h,k} \cdot \nabla
    \lambda_h \dd x \right] \notag  + r \\
  & = \mexp \left[ \int_\Gamma (a_{h,k} - a_h) \nabla u_{h,k} \cdot
    \nabla \lambda_h \dd x \right] + r \notag,
\end{align*}
where $r=\mexp \left[ \int_\Gamma (f-f_{h,k})\lambda_h \dd x \right]$
is a remainder term appearing in equation \eqref{eq:rlimit}.  Looking
at the final expression, we expand $(a_{h,k} - a_h)$ in a telescoping
series in $k$,
\begin{equation}
  \label{eq:quad-telescoping}
  \begin{split}
    &\left\lvert \mexp \left[ \int_\Gamma (a_{h,k} - a_h) \nabla
        u_{h,k} \cdot \nabla
        \lambda_h \dd x \right] \right\rvert \\
    &\quad \leq \sum_{j=0}^\infty \left| \mexp \left[ \int_\Gamma
        (a_{h, k 2^{-j}} - a_{h, k 2^{-(j+1)}})
        \nabla u_{h,k} \cdot \nabla \lambda_h \dd x \right] \right| \\
    &\quad \leq \sum_{j=0}^\infty C k^\gamma 2^{-\gamma j} = \frac{C
      k^\gamma}{1-2^{-\gamma}},
  \end{split}
\end{equation}
and the result follows by assumption \eqref{eq:rlimit} and introducing
a constant $\hat C>C$ such that
\begin{equation*}
  \frac{C k^\gamma}{1-2^{-\gamma}} + \left| r \right| \leq \frac{\hat C
    k^\gamma}{1-2^{-\gamma}}. \qquad \endproof
\end{equation*}

Theorem~\ref{thm:quadScales} motivates estimating the observable for
the quadrature error by an MC approximation of
\begin{equation}
  \label{eq:quad-decay-rate}
  \mexp \left[ \int_\Gamma (a_{h,k} - a_{h,k/2})\nabla 
    u_{h,k} \cdot \nabla \lambda_{h} \dd x \right] \approx C k^\gamma.
\end{equation}
We remark that, as was the case for Galerkin estimator, here we must
also assume that the expected quadrature estimator is bounded from
below to formally replace $C k^\gamma$ in \eqref{eq:quad-telescoping}
by \eqref{eq:quad-decay-rate} when estimating the observable for the
expected quadrature error. However, unlike in the case of the Galerkin
error, taking absolute values in the integrand produces an estimator
that is too large to be useful in practice. This is due to the
presence of cancellations in the observable for the quadrature error
even for generic choices of observable $g$.

Putting off the justification for the assumption on scales until
Theorem~\ref{thm:quadScalesMotivation}, we proceed to test the
estimator for the one-dimensional model problem
\eqref{eq:model-problem} for the observable $g=1$.  For a given
sample, $\omega_m \in \Omega$, we compute the pathwise estimator
\begin{equation*}
  Q^{h,k}_m := \frac{h^2}{2} 
  \sum_{K_h} (a_{h,k}(\omega_m) - a_{h,k/2}(\omega_m)) 
  \nabla u_{h,k}(\omega_m) \cdot \nabla\lambda_{h,\tilde{k}}(\omega_m),
\end{equation*}
where the sum is over $h$-elements, $K_h$, and $\tilde{k}$ is a
reference mesh that will receive further attention in the discussion
below. In dimension one, we use the trapezoid rule, for convenience,
and, in particular, for $k = h$ we have
\begin{equation*}
  a_{h,h}(x) = \frac{1}{2} (a(x_j) + a(x_{j+1})),
\end{equation*}
for $x \in K_{h} = [x_{j}, x_{j+1}]$, the (spatial) mean value of $a$
on the $h$-element, $K_h$. The estimator requires $a_{h,k}$ at two
levels of the quadrature mesh: $k = h2^{-n}$ and $k/2 =
h2^{-(n+1)}$. We then form the MC approximation,
\begin{equation*}
  \hat{Q}^{h,k} := \frac{1}{M} \sum_{m = 1}^{M} Q^{h,k}_m,
\end{equation*}
for a sample population of size $M$. When necessary, we distinguish
between different samples of the MC approximations for the error and
the estimator with the addition of a subscript.

Taking a second look at Figure~\ref{fig:quad-rate}, we compare the
quadrature error and the estimator for the observable $g= 1$ for the
one-dimensional model problem \eqref{eq:model-problem}. We compute MC
approximations, $\hat{\mathcal{Q}}^{h,h}$ and $\hat{Q}^{h,h}$, for
$M = 2^{13}$ samples to ensure that the statistical error, which is on
the same order as the quantity $\sigma_M := \sigma_s/ \sqrt{M}$, is
negligible compared to the quadrature error (see
Table~\ref{tbl:sample-variance}). The error observable uses
$u_h = u_{h,\tilde{k}}$ on a reference quadrature mesh of
$\tilde{k} = 2^{-24}$ and the estimator is computed using
$\lambda_h = \lambda_{h,\tilde{k}}$ on a reference quadrature mesh of
$\tilde{k} = 2^{-22}$. Here, for both the error and the estimator, we
take the quadrature mesh to match the finite element mesh so that we
can test the sensitivity of the dual solution in the estimator on the
reference quadrature mesh. Thus, in Figure~\ref{fig:quad-rate}, we
also plot $\hat{Q}^{h,h}_*$, a more computationally practical
estimator, computed using the discrete dual $\lambda_{h,h/2}$, with a
quadrature mesh of size $\tilde{k} = h/2$, instead of the reference
mesh. We observe that there is no perceivable difference between the
estimators $\hat{Q}^{h,h}$ and $\hat{Q}_*^{h,h}$. Therefore, we
redefine $\hat{Q}^{h,k} := \hat{Q}_*^{h,k}$ and employ the practical
estimator, using the discrete dual on the quadrature mesh,
$\tilde{k} = k/2$, in the sequel. Note the constant $C = 2$ is not
chosen as a best fit but merely to demonstrate that the estimator
provides a reliable estimate using a constant factor.
     
\begin{figure}[]
  \centering
  \includegraphics[width=\sfac]{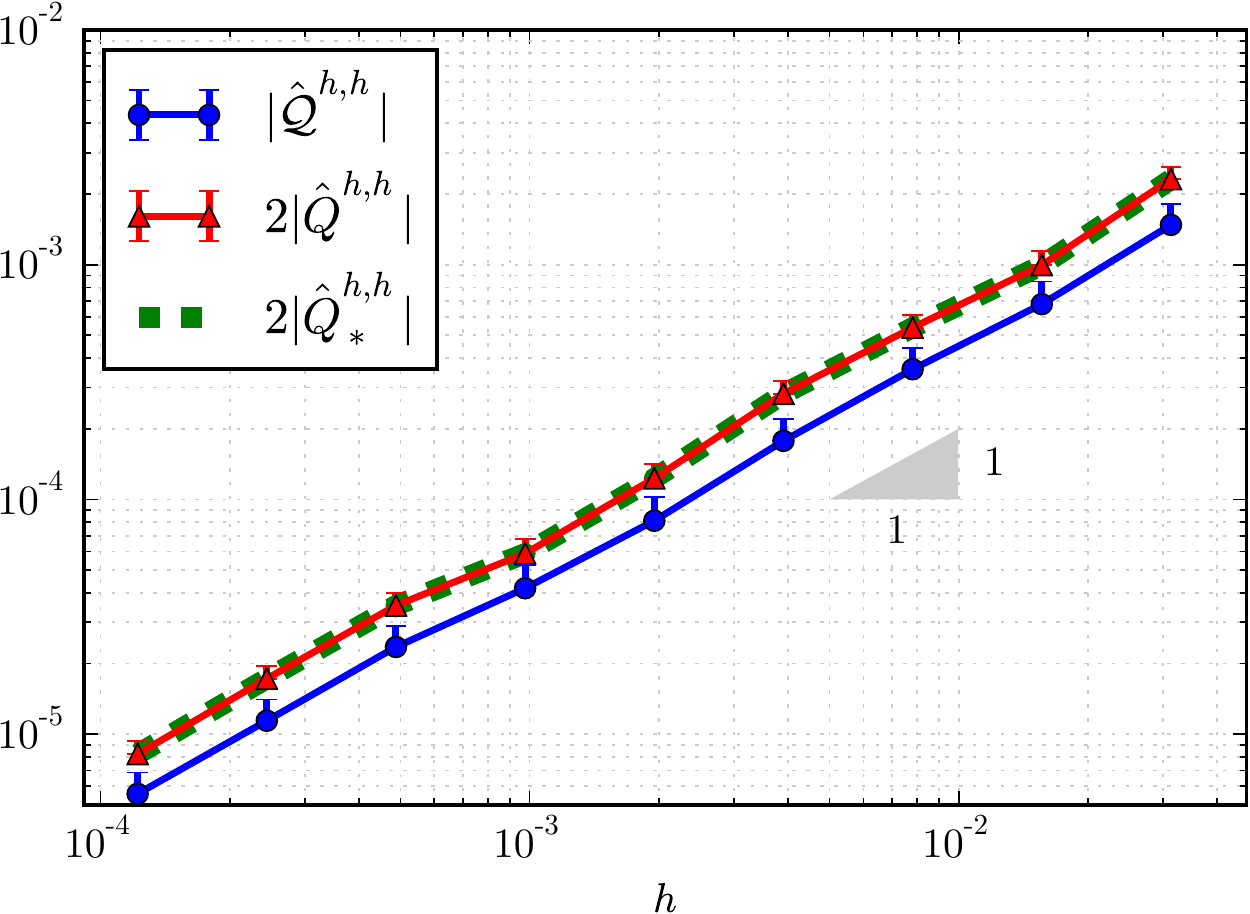}
  \caption{For problem~\eqref{eq:model-problem} for the observable
    $g=1$, the expected quadrature error committed in the finite
    element method can be estimated by local error indicators. Here
    the calculation of the mean is based on $M = 2^{13}$ samples to
    ensure that the statistical error is negligible and the one-sided
    error bars indicate $5$ $\sigma_M$-deviations from the
    corresponding mean.}
  \label{fig:quad-rate}
\end{figure}

\begin{table}[h]
  \footnotesize
  \centering
  \caption{The statistical error as measured by $\sigma_M  := \sigma_s/ \sqrt{M}$, based on $M = 2^{13}$ samples to ensure
    that the statistical error is negligible, corresponding to $\hat{\mathcal{Q}}^{h,h}$ (third row) and $\hat{Q}^{h,h}$ (fifth rows) for the observable $g=1$, is small compared to the corresponding expected quadrature error.}
  \label{tbl:sample-variance}
  \begin{tabular}{|l|rrrrrrrrr|}
    \hline
    $\log_2 h$     &  -13        & -12        & -11        & -10        & -9         & -8         & -7         & -6         & -5         \\
    \hline
    $\hat{\mathcal{Q}}^{h,h}$ & 5.6e-6 & 1.1e-5 & 2.4e-5 & 4.2-5 & 8.1e-5 & 1.8e-4 & 3.6e-4 & 6.8e-4 & 1.5e-3 \\
    $\sigma_M$ &  2.6e-7 & 5.3e-7 & 1.1e-6 & 2.1e-6 & 4.2e-6 & 8.6e-6 & 1.7e-5 & 3.4e-5 & 6.7e-5 \\
    \hline
    $\hat{Q}^{h,h}$ & 4.1e-6 & 8.6e-6 & 1.8e-5 & 2.9e-5 & 6.1e-5 & 1.4e-4 & 2.7e-4 & 5.0e-4 & 1.2e-3 \\
    $\sigma_M$ &  2.3e-7 & 4.6e-7 & 9.3e-7 & 1.8e-6 & 3.4e-6 & 7.4e-6 & 1.5e-5 & 2.9e-5 & 5.8e-5 \\
    \hline
  \end{tabular}
\end{table}

Next, we proceed to justify the assumption on scales for the
quadrature rule. One can use Malliavin calculus arguments to show that
\eqref{eq:assumption-on-scales-quadrature} has the desired decay rate
(see the proof of Theorem 3.3 in \cite{SzepessyTemponeZouraris:2001},
that provides an estimate for an adaptive time-stepping method for
SDE, containing a similar argument). In the setting of model problem
\eqref{eq:model-problem}, the argument greatly simplifies if one takes
advantage of the available representations of the primal and dual
solutions.

\begin{theorem}\label{thm:quadScalesMotivation}
  Consider the model problem~\eqref{eq:model-problem} and the
  corresponding finite element solution $u_{h,k}$ defined by equation
  \eqref{eq:forward-with-quad}. Assume that the observable $g$ is
  sufficiently smooth to ensure that $G \in L^1(0,1)$, for the
  primitive function $G$ as defined in \eqref{eq:Gdef}.  Then, the
  quadrature error assumption of scales
  \begin{equation*}
    \left| \mexp \left[ \int_\Gamma (a_{h,k} - a_{h, k/2}) \nabla
        u_{h,k} \cdot \nabla \lambda_{h} \dd x \right] \right| 
    < C k^\gamma,
  \end{equation*}
  holds with $\gamma =1$, provided the forward Euler quadrature rule,
  \eqref{eq:forwardEuler} below, is used to compute the stiffness
  matrix function $a_{h,k}$.
\end{theorem}

\begin{proof}
  One can represent an explicit Euler quadrature rule so that
  \begin{equation*}
    a_{h,h}(x) = a(x_i),
  \end{equation*}
  for $x \in [x_i, x_{i+1})$, and, more generally,
  \begin{equation}\label{eq:forwardEuler}
    a_{h,k}(x) = \sum_{j=1}^{h/k} \frac{k}{h} a(x_{i,j}),
  \end{equation}
  for $x \in [x_i, x_{i+1})$ and $x_{i,j} := ih+jk$. As alternatives
  to $a_{h,k}$, which are piecewise constant functions on intervals of
  length $h$, it will be sufficient to consider functions
  $\bar{a}_{h,k}$ that are piecewise constant on intervals of length
  $k$ such that
  \begin{equation*}
    \int_{0}^1 (a_{h,k}- \bar a_{h,k})\nabla u_h \cdot \nabla v_{h} \dd x = 0 
  \end{equation*}
  for all $u,v \in V_h$.  For explicit Euler quadrature, we obtain
  that
  \begin{equation*}
    \bar a_{h,k}(x) = e^{ \bar{B}_{h,k}(x) }, 
  \end{equation*}
  with
  \begin{equation*}
    \bar B_{h,k}(x) = \sum_{i=1}^{h^{-1}} \sum_{j=1}^{h/k} \ind{\{x_{i,j} < x\}} \Delta B_{i,j}^{k},
  \end{equation*}
  where $\Delta B_{i,j}^{k}$ is an increment of the Brownian bridge
  over the interval $[ih +(j -1)k, ih+jk)$. Thus, $\bar{a}_{h,k}(x)$
  is piecewise constant over $k$-intervals due to the piecewise
  constant $\bar{B}_{h,k}(x)$.
  Similarly, we have that
  \begin{equation*}
    \bar{a}_{h,k/2}(x) = e^{ \bar{B}_{h,k/2}(x) } 
  \end{equation*}
  with
  \begin{equation*}
    \bar{B}_{h,k/2}(x) = \sum_{i=1}^{h^{-1}} \sum_{j=1}^{2h/k} \ind{\{ih+jk/2 < x\}} \Delta B_{i,j}^{k/2}.
  \end{equation*}
  For working on the pointwise difference between
  $\bar{a}_{h,k/2} - \bar{a}_{h,k}$, we simplify further by reducing
  to one index, $\ell(i,j) = i 2h/k+ j$, and noting that we may write
  \begin{equation*}
    \bar{B}_{h,k}(x) = \sum_{\ell \leq 2k^{-1}} \ind{\{\ell k/2 < x\}} \ind{\{\ell \text{ is even}\}}
    ( \Delta B_{\ell-1}^{k/2} + \Delta B_{\ell}^{k/2})
  \end{equation*}
  and
  \begin{equation*}
    \bar{B}_{h,k/2}(x) = \sum_{\ell \leq 2k^{-1}} \ind{\{\ell k/2 < x\}} 
    \Delta B_{\ell}^{k/2}.
  \end{equation*}
  Consequently,
  \begin{equation*}
    (\bar{B}_{h,k} - \bar{B}_{h,k/2})(x) = \begin{cases}
      -\Delta B^{k/2}_{\lfloor 2x/k \rfloor} 
      & \text{if } \lfloor 2x/k \rfloor \text{ is odd.}\\ 
      0 
      & \text{else}.
    \end{cases}
  \end{equation*}
  Below we shall test the sensitivity of terms arising in the
  representation of the quadrature error with respect to the
  increments of the processes defined above. When convenient, we shall
  abuse the notation slightly by writing
  $a(\bar{B}_{h,k}(x)) = e^{\bar{B}_{h,k}(x)}$ for the log-Brownian
  bridge path depending on the discrete increments,
  $\bar{B}_{h,k}(x)$.

  Using this newly introduced notation, we apply a Taylor expansion
  and the Mean Value Theorem to the left-hand side of
  \eqref{eq:quad-decay-rate} to obtain
  \begin{equation}
    \label{eq:quad-decay-rate-expanded}
    \begin{split}
      \mexp &\left[ \int_0^1 \left(a(\bar{B}_{h,k}(x)) - a(\bar{B}_{h,k/2}(x))\right)  u_{h,k}' \lambda'_h \dd x \right] \\
      &= \mexp \left[ \int_0^1  a'(\bar{B}_{h,k/2}) u_{h,k}' \lambda'_h \left(\bar{B}_{h,k} - \bar{B}_{h,k/2}\right) \dd x \right] \\
      &\quad + \mexp \left[ \frac{1}{2} a''(\theta \bar{B}_{h,k/2} +
        (1-\theta)\bar{B}_{h,k}) u_{h,k}' \lambda'_h
        \left(\bar{B}_{h,k} - \bar{B}_{h,k/2}\right)^2 \dd x \right]
    \end{split}
  \end{equation}
  where $\theta = \theta(x,\omega) \in [0,1]$ denotes a Mean Value
  Theorem constant.  Assuming all needed moments are bounded (a
  reasonable assumption since, under the slight abuse of notation
  $a(B) = e^B$, we have that $a'' = a$ in the case of our simple model
  problem), the second summand on the right-hand side
  of~\eqref{eq:quad-decay-rate-expanded} may be bounded by H\"older's
  inequality as follows
  \begin{equation*}
    \begin{split}
      \int_0^1 \frac{1}{2} & \mexp \left[a^{\prime\prime}(\theta
        \bar{B}_{h,k/2} + (1-\theta)\bar{B}_{h,k}) \left(\bar{B}_{h,k}
          - \bar{B}_{h,k/2}\right)^2
        u_{h,k}' \lambda'_h \right] \dd x  \\
      & \leq \frac{1}{2}
      \int_0^1 \!\!\!  \sqrt{\mexp \left[ \left(a^{\prime\prime}(\theta \bar{B}_{h,k/2} + (1-\theta)\bar{B}_{h,k}) u_{h,k}' \lambda'_h\right)^2\right]} \sqrt{\mexp \left[ \left(\bar{B}_{h,k} - \bar{B}_{h,k/2}\right)^4 \right]} \dd x \\
      & \leq C k \int_0^1 |G_h| \dd x \leq C k
    \end{split}
  \end{equation*}
  where we have used
  \begin{equation*}
    \mexp \left[ (\bar{B}_{h,k}(x) - \bar{B}_{h,k/2}(x))^4 \right] 
    = \mexp \left[ \left(\Delta B^{k/2}_{\lfloor 2x/k \rfloor} \right)^4 
      \ind{\{ \lfloor 2x/k \rfloor \text{ is odd}\}} \right]  = O(k^2),
  \end{equation*}
  that $\lambda_h' = G_h/a_h$, cf.~Lemma~\ref{lem:discretesolutions},
  and that $G_h \in L^1(0,1)$, which implicitly follows from
  $G \in L^1(0,1)$.

  For achieving a bound of the same order for the first summand
  of~\eqref{eq:quad-decay-rate-expanded}, we begin by noting that
  \begin{equation}
    \label{eq:quad-decay-first-summand}
    \begin{split}
      \int_0^1 &\mexp \left[ a'(\bar{B}_{h,k/2}) u_{h,k}' \lambda'_h (\bar{B}_{h,k} - \bar{B}_{h,k/2}) \right] dx \\
      & = \int_0^1 \mexp \left[ a'(\bar{B}_{h,k/2}) u_{h,k}' \mexp [
        \lambda'_h \mid \Delta B^{k/2}_{\lfloor 2x/k \rfloor}] \Delta
        B^{k/2}_{\lfloor 2x/k \rfloor} \ind{\{\lfloor 2x/k \rfloor
          \text{ is odd}\}} \right] \dd x.
    \end{split}
  \end{equation}
  We introduce the notation
  \begin{equation*}
    S\left( (\Delta B^{k/2}_{\ell})_{\ell=1}^{k^{-1}} \right) := a'(\bar{B}_{h,k/2})
    u_{h,k}' \mexp \left[ \lambda'_h \mid \Delta B^{k/2}_{\lfloor 2x/k \rfloor} \right]
  \end{equation*}
  with the convenient shorthand
  \begin{equation*}
    \hat{S}(z) = S(\Delta B^{k/2}_{1}, \ldots,\Delta B^{k/2}_{\lfloor
      2x/k \rfloor -1 }, z, \Delta B^{k/2}_{\lfloor 2x/k \rfloor +1 },
    \ldots,\Delta B^{k/2}_{k^{-1}}),
  \end{equation*}
  and Taylor expand, with, once again,
  $\theta = \theta(x,\omega) \in [0,1]$ denoting a Mean Value Theorem
  coefficient,
  \begin{equation*}
    \begin{split}
      S\left( (\Delta B^{k/2}_{\ell})_{\ell=1}^{k^{-1}} \right)
      & = \hat{S}(\Delta B^{k/2}_{\lfloor 2x/k \rfloor} ) \\
      &= \hat{S}(0) + \hat{S}'(\theta \Delta B^{k/2}_{\lfloor 2x/k
        \rfloor}) \Delta B^{k/2}_{\lfloor 2x/k \rfloor}.
    \end{split}
  \end{equation*}
  Substituting the expansion of
  $\hat{S}(\Delta B^{k/2}_{\lfloor 2x/k \rfloor} )$
  into~\eqref{eq:quad-decay-first-summand}, noting that $\hat{S}(0)$
  is independent of $\Delta B^{k/2}_{\lfloor 2x/k \rfloor}$, and using
  H\"older's inequality on the other integrand yields
  \begin{equation*}
    \begin{split}
      \int_0^1 &\mexp \left[ a'(\bar{B}_{h,k/2}) u_{h,k}' \mexp \left[
          \lambda'_h \mid \Delta B^{k/2}_{\lfloor 2x/k \rfloor}\right]
        \Delta B^{k/2}_{\lfloor 2x/k \rfloor} \ind{\{\lfloor 2x/k \rfloor \text{ is odd}\}} \right] \dd x  \\
      &= \int_0^1 \mexp \left[ (\hat{S}(0) + \hat{S}'(\theta \Delta B^{k/2}_{\lfloor 2x/k \rfloor}) \Delta B^{k/2}_{\lfloor 2x/k \rfloor} \Delta B^{k/2}_{\lfloor 2x/k \rfloor} \ind{\{\lfloor 2x/k \rfloor \text{ is odd}\}} \right] \dd x \\
      & \leq  \int_0^1 \mexp \left[ \hat{S}(0)\right] \mexp \left[ \Delta B^{k/2}_{\lfloor 2x/k \rfloor} \ind{\{\lfloor 2x/k \rfloor \text{ is odd}\}} \right] \dd x \\
      & \qquad + \int_0^1 \sqrt{ \mexp \left[ \left(\hat{S}'(\theta
            \Delta B^{k/2}_{\lfloor 2x/k \rfloor})\right)^2\right] }
      \sqrt{ \mexp \left[ \left(\Delta B^{k/2}_{\lfloor 2x/k \rfloor}\right)^4 \ind{\{\lfloor 2x/k \rfloor \text{ is odd}\}} \right] } \dd x \\
      &\leq C k,
    \end{split}
  \end{equation*}
  since
  $\mexp [ \Delta B^{k/2}_{\lfloor 2x/k \rfloor} \ind{\{\lfloor 2x/k
    \rfloor \text{ is odd}\}} ] = 0$.
  Hence the assumption on scales for the quadrature error holds with
  $\gamma=1$ in the case of the simple model problem
  \eqref{eq:model-problem}.
\end{proof}

%
\section{Discussion and future work}
\label{sec:discussion}
%
In the present work, we have presented reliable and computable
estimates for the Galerkin and quadrature errors committed in
observables of the finite element approximation of a class of elliptic
PDE with rough stochastic conductivities. These estimators, based on
local error indicators, were obtained by first making an assumption on
the scales of the model problem and then utilizing a simple
telescoping argument. For the model problems under consideration, we
demonstrated that the assumptions are satisfied. Even in the case of a
simple model problem, the standard theory for constructing \emph{a
  posteriori} error estimates failed to give reliable estimators in
this setting. Therefore, the error estimates given here fill a much
needed gap by providing an important and novel computational tool for
this class of problems.

In this direction, these estimates are useful for constructing
adaptive algorithms. These estimates are a step towards the
construction of adaptive algorithms for variance reduction techniques
for MC methods for PDE with rough stochastic coefficients. As
mentioned in the introduction, the final level stopping criterion in a
MLMC, MIMC, or CMLMC method can be built upon the error estimates
derived here and a minor extension of the present theory would render
it possible to also estimate mean square errors between numerical
solutions on consecutive resolution levels.
 
%
\bibliographystyle{siam}%
\bibliography{SISC_M104426.bib}%
%

\end{document}